\documentclass[11pt,psamsfonts]{amsart}
\usepackage{palatino}
\usepackage{euler,color}
\usepackage[all]{xy}
\usepackage{amssymb,amscd,amsmath,amsthm,dsfont}

\chardef\bslash=`\\ 

\newtheorem{theorem}{Theorem}[section]
\newtheorem{corollary}[theorem]{Corollary}
\newtheorem{lemma}[theorem]{Lemma}
\newtheorem{proposition}[theorem]{Proposition}
\newtheorem{prop}[theorem]{Proposition}

\theoremstyle{remark}
\newtheorem{remark}[theorem]{Remark}
\newtheorem{remarks}[theorem]{Remarks}
\newtheorem{example}[theorem]{Example}

\theoremstyle{definition}

\newtheorem{convention}[theorem]{Convention}
 
\numberwithin{equation}{section}

\newcommand{\thmref}[1]{Theorem~\ref{#1}}
\newcommand{\secref}[1]{Section~\ref{#1}}
\newcommand{\proref}[1]{Proposition~\ref{#1}}
\newcommand{\lemref}[1]{Lemma~\ref{#1}}
\newcommand{\corref}[1]{Corollary~\ref{#1}}
\newcommand{\remref}[1]{Remark~\ref{#1}}
\newcommand{\clsp}{\overline{\operatorname{span}}}
\newcommand{\lsp}{\operatorname{span}}

\newcommand{\Aut}{\operatorname{Aut}}

\newcommand{\hatz}{\widehat\Z}
\newcommand{\cq}{\mathcal{C}_{\Q}}

\newcommand{\nx}{\mathbb N^{\times}}
\newcommand{\inv}{^{-1}}
\def\sn{\mathcal N}

\newcommand{\nxnx}{{\mathbb N \rtimes \mathbb N^\times}}
\newcommand{\qxqx}{{\mathbb Q \rtimes \mathbb Q^*_+}}
\newcommand{\N}{\mathbb N}
\newcommand{\qn}{\mathcal Q_\mathbb N}

\newcommand{\Z}{\mathbb Z}
\newcommand{\Q}{\mathbb Q} 
\newcommand{\qx}{\mathbb Q^*_+}
\newcommand{\C}{\mathbb C}
\newcommand{\R}{\mathbb R}
\newcommand{\T}{\mathbb T}
\newcommand{\TT}{\mathcal T}
\newcommand{\HH}{\mathcal H}

\newcommand{\OO}{\mathcal O}

\newcommand{\primes}{\mathcal P}

\newcommand{\bcheck}{\mathcal C_{\Q}}

\def\lcm{\operatorname{lcm}}

\def\her{\operatorname{Her}}
\def\tr{\operatorname{Tr}}

\newcommand{\notdiv}{\nmid}

\def\oldz{r}

\def\newz{z}

\def\setb{E}

\title[Phase transition on a Toeplitz algebra]{Phase transition on the Toeplitz algebra of\\
the affine semigroup over the natural numbers}
\date{21 July 2009}
\thanks{This research was supported by the Natural Sciences and Engineering Research
Council of Canada and by the Australian Research Council.}

\author[Marcelo~Laca]{Marcelo Laca}
\address{Department of Mathematics and Statistics, University of
Victoria, Victoria, BC V8W 3P4, Canada}
\email{laca@math.uvic.ca}
\author[Iain~Raeburn]{Iain Raeburn}
\address{School of Mathematics and Applied Statistics,
University of Wollongong, NSW  2522, Australia}
\email{raeburn@uow.edu.au}

\begin{document}
\begin{abstract} We show that the group $\qxqx$ of orientation-preserving affine transformations of the
rational numbers is quasi-lattice ordered by its subsemigroup $\nxnx$.
 The associated Toeplitz $C^*$-algebra $\TT(\nxnx)$ is universal for
isometric representations which are covariant in the sense of Nica.
We give a presentation of $\TT(\nxnx)$ in terms of generators and relations, and use this to show that the $C^*$-algebra $\qn$ recently introduced by
Cuntz is the boundary quotient of $(\qxqx, \nxnx)$ in the sense of
Crisp and Laca. The Toeplitz algebra $\TT(\nxnx)$ carries a natural dynamics $\sigma$, which induces the one considered by Cuntz on the quotient
$\qn$, and our main result is
 the computation of  the KMS$_\beta$ (equilibrium) states of the
dynamical system $(\TT(\nxnx), \R,\sigma)$
 for all values of the inverse temperature $\beta$.
  For $\beta \in [1, 2]$ there is a unique
 KMS$_\beta$ state, and the KMS$_1$ state factors through the quotient map onto $\qn$, giving the unique KMS state discovered by Cuntz. At
$\beta =2$ there is a phase transition, and for $\beta>2$ the KMS$_\beta$ states are indexed
by probability measures on
 the circle. There is a further phase transition at $\beta=\infty$, where the KMS$_\infty$ states are indexed by the probability measures on the
circle, but the ground states are indexed by the states on the classical Toeplitz algebra~$\TT(\N)$.
 \end{abstract}
\maketitle
\section*{Introduction}
Cuntz has recently introduced and studied a $C^*$-algebra $\qn$ which is generated by an isometric representation of the semidirect product $\nxnx$ of the additive semigroup $\N$ by the natural action of the multiplicative semigroup $\N^\times$ \cite{cun2}. Cuntz proved that $\qn$ is simple and purely infinite, and that it admits a natural dynamics for which there is a unique KMS state, which occurs at inverse temperature $1$. He also showed that $\qn$ is closely related to other very interesting $C^*$-algebras, such as the Bunce-Deddens algebras and the Hecke $C^*$-algebra of Bost and Connes~\cite{bos-con}.

In recent years there has been a great deal of interest in other $C^*$-algebras generated by isometric representations of a semigroup $P$, such as the Toeplitz algebra $\TT(P)$ which is generated by the canonical isometric representation on $\ell^2(P)$, and it is natural to ask how Cuntz's algebra relates to these other $C^*$-algebras. It is obviously not the same as $\TT(\nxnx)$: in $\qn$, the isometry associated to the additive generator is unitary. So one is led to guess that $\qn$ might be a quotient of $\TT(\nxnx)$. If so, the Toeplitz algebra $\TT(\nxnx)$ looks very interesting indeed. There is a general principle, going back at least as far as  \cite{eva} and used to effect in \cite{EL,lacaN},  which suggests that the Toeplitz algebra should have a much richer KMS structure than its simple quotient. (The principle is illustrated by the gauge action on the Cuntz algebra $\OO_n$, where the Toeplitz-Cuntz analogue $\TT\OO_n$ has KMS states at every inverse temperature $\beta\geq \log n$, but only the one with $\beta=\log n$ factors through a state of $\OO_n$.)

Our goal here is to answer these questions. We first prove that the pair consisting of the semigroup $\nxnx$ and its enveloping group $\qxqx$ form a quasi-lattice ordered group in the sense of Nica \cite{nica}; this is itself a little surprising, since it is not one of the semi-direct product quasi-lattice orders discussed in \cite{nica}. However, once we have established that $(\qxqx,\nxnx)$ is quasi-lattice ordered, it follows
that the Toeplitz algebra $\TT(\nxnx)$ is universal for Nica-covariant isometric representations of $\nxnx$. We can then run this pair through the general theory of \cite{CL2}, and with some effort we can recognise $\qn$ as the boundary quotient of $\TT(\nxnx)$. From this we deduce that $\qn$ is a partial crossed product, and thereby provide another proof that it is purely infinite and simple. 

We then consider a natural dynamics $\sigma$ on $\TT(\nxnx)$ which induces that
studied by Cuntz on $\qn$, and compute the simplices of KMS${}_\beta$
states for every inverse temperature $\beta$. For
$\beta>2$ the KMS${}_\beta$ states are parametrised by probability
measures on the unit circle. For $\beta \in [1,2]$, only the one
corresponding to Lebesgue measure on the circle survives, and there is
a unique KMS${}_\beta$ state. This phase transition is associated to
the pole of the
partition function, which is the shifted Riemann zeta function $\zeta
(\beta-1)$.

Our system $(\TT(\nxnx),\R,\sigma)$  therefore exhibits some of the behaviour of other number-theoretic systems \cite{bos-con,diri, HL, CM,LvF, cmgl2}, even though  our system lacks some of the features which make
the other number-theoretic systems so interesting, such as the
presence of a large symmetry group like
the idele class group of $\Q$ in \cite{bos-con} or
its two-dimensional analogue in \cite{CM}.
However, the KMS states in our system also display
several interesting phenomena which have not previously occurred for
dynamical systems of number-theoretic origin.
First, not all KMS states factor through the expectation onto the
commutative subalgebra spanned by the range projections of the
isometries: for $\beta>2$, the KMS$_\beta$ states do not necessarily
vanish on the additive generator, which for this system is fixed by
the dynamics. Second, the unique KMS${}_\beta$ states for $1\leq \beta
\leq 2$ have a circular symmetry,  which is broken at $\beta=2$, but
this symmetry does not come from an action of the circle on the $C^*$-
algebra $\TT(\nxnx)$. This phenomenon appears to be related to the
fact that the enveloping group $\qxqx$ is nonabelian, and the dual
coaction of $\qxqx$ on $\TT(\nxnx)$ cannot be ``restricted'' to a
coaction of the additive subgroup $\Q$. And third, as foreshadowed in
\cite[Definition 3.7]{CM2}, the set of
KMS$_\infty$ states (the states that are limits of KMS$_\beta$ states
as $\beta \to \infty$),
which is isomorphic to the simplex of probability measures on the
circle, is much smaller than the set of ground states, which is
isomorphic to the state space of the classical Toeplitz algebra, and
hence is not a simplex.

We begin our paper with a brief discussion of notation and preliminaries from quasi-lattice ordered groups and number theory. Then in \secref{semigpisql}, we show that the semigroup $\nxnx$ induces a quasi-lattice order on the group $\qxqx$,
and deduce that the associated Toeplitz $C^*$-algebra is generated by a universal Nica-covariant isometric  representation (\corref{Toepl=univ}). 
In \secref{euclideanalgorithm} we work out a version of the euclidean algorithm suitable for computations involving 
Nica-covariant isometric representations of $\nxnx$. Once this is done, we characterise in \secref{secpresentation} the Toeplitz
 $C^*$-algebra $\TT(\nxnx) $ of $\nxnx$ by giving a presentation in terms of generators and relations (\thmref{toeplitzpresentation});
 some of the relations are recognizably variants on Cuntz's relations for $\qn$, but others are new.  
 
 To apply the structure theory of Toeplitz $C^*$-algebras of quasi-lattice orders,  we need to understand the Nica spectrum of $\nxnx$, and in \secref{nicaspectrum} we give an explicit parametrisation of this spectrum using  integral adeles and supernatural numbers. This allows us to identify the \emph{boundary} of the spectrum, as defined in \cite{CL2}. The boundary in \cite{CL2} is the smallest of many boundaries: there are many ways one can ``go to infinity'' in the semigroup $\nxnx$. Of particular interest is the {\em additive boundary}, which corresponds to going to infinity along the additive semigroup $\N$. In \proref{prodstructure} we show that the additive boundary  has a direct product decomposition, which later plays a crucial role in our construction and analysis of KMS$_\beta$ states. In \secref{sectionqn}, we prove that Cuntz's $\qn$ is isomorphic to the \emph{boundary quotient} studied in \cite{CL2} (Theorem~\ref{qnisboundaryquotient}), and we use the theory developed in \cite{CL2} to give a quick proof that $\qn$ is simple and purely infinite.

 In \secref{secKMS} we introduce a natural dynamics $\sigma$ on $\TT(\nxnx)$,  
  and state our main result, which
describes the phase transition associated to this natural dynamics (\thmref{maintheorem}). 
We also discuss the significance of this phase transition in relation to the symmetries and the structure of the $C^*$-algebra $\TT(\nxnx)$.
We begin the proof of the main theorem in \secref{seccharacterisationKMSground}. We first show that there are no KMS states for $\beta <1$,  and then we characterise the KMS$_\beta$ states by their behaviour on a spanning set for $\TT(\nxnx)$. This characterisation implies that a KMS$_\beta$ state is determined by its behaviour on the $C^*$-subalgebra generated by the additive generator (\lemref{KMScharacterisationlemma}). In \lemref{lemmagroundcharacterisation}, we give a similar characterisation of ground states.

In \secref{secconstructionKMSground}, we construct KMS$_\beta$ states for $\beta \in [1,\infty]$ by inducing a probability measure 
on the additive boundary to a state of $\TT(\nxnx)$ via the conditional expectation of the dual coaction of  $\qxqx$ (\proref{productmeasure}). 
In  \proref{constructKMS>2}, we consider $\beta \in (2,\infty]$, and give a spatial construction of a KMS$_\beta$ state for each probability measure on the circle.
A parallel construction also yields a ground state for each state of $\TT(\N)$.
We complete the proof of \thmref{maintheorem} in \secref{secsurjectivity}, by showing  that the explicit 
constructions of  \secref{secconstructionKMSground}
correspond exactly to the possibilities left open in \secref{seccharacterisationKMSground}. 
The interesting  case here is  $\beta\in [1,2]$,
and our key idea is the reconstruction formula of 
\lemref{phirestrictedtoQB}, which was inspired by Neshveyev's ergodicity proof of the uniqueness of KMS$_\beta$ states on the Hecke $C^*$-algebra of Bost and Connes \cite{nes}. Curiously, though, we can now see that the analogous reconstruction formula for the Bost-Connes system does not need ergodicity at all. We give this ``ergodicity-free" version of Neshveyev's proof in Appendix~ \ref{bcuniqueness}.

\section{Notation and preliminaries} \label{notationandpreliminaries}

\subsection{Toeplitz algebras}\label{Toeplitzalgs} Every cancellative semigroup $P$ has a natural \emph{Toeplitz representation} $T:P\to B(\ell^2(P))$, which is characterised in terms of the usual basis $\{e_x:x\in P\}$ by $T_ye_x=e_{yx}$. Notice that the operators $T_y$ are all isometries. The \emph{Toeplitz algebra} $\TT(P)$ is the $C^*$-subalgebra of $B(\ell^2(P))$ generated by the operators $\{T_y\}$. Our semigroups $P$ will always be generating subsemigroups of a group $G$; as a motivating example, consider the subgroup $\N^2$ of $\Z^2$. Any isometric representation $V$ of $\N^2$ on Hilbert space is determined by the two commuting isometries $V_{(1,0)}$ and $V_{(0,1)}$. In the Toeplitz representation of $\N^2$, however, the two generators satisfy the extra relation $T_{(1,0)}^*T_{(0,1)}=T_{(0,1)}T_{(1,0)}^*$, and it turns out that this relation uniquely characterises the Toeplitz algebra $\TT(\N^2)$ among $C^*$-algebras generated by non-unitary isometric representations of $\N^2$. Nica's theory of quasi-lattice ordered groups $(G,P)$ provides a far-reaching generalisation of this result.

A submonoid $P$ of a group $G$ naturally induces a left-invariant partial order on  by $x\leq y$ iff $x\inv y \in P$.
Following Nica \cite{nica}, we say that $(G,P)$ is {\em quasi-lattice ordered}
if every pair of elements $x$ and $y$ in $G$ which have a common upper bound in $G$ have a least upper bound $x\vee y $. When they have a common upper bound we write $x\vee y < \infty$, and otherwise $x\vee y=\infty$. (This is not quite Nica's original definition, but it is equivalent. This and other reformulations are discussed in \cite[\S3]{CL1}.) An isometric representation $V:P\to B(\HH)$ is \emph{Nica covariant} if 
\[
V_xV_x^*V_yV_y^*=\begin{cases}V_{x\vee y}V_{x\vee y}^*&\text{if $x\vee y<\infty$}\\
0&\text{if  $x\vee y=\infty$,}\end{cases}
\]
or equivalently, if
\begin{equation}\label{genNicacov}
V_x^*V_y=\begin{cases}V_{x^{-1}(x\vee y)}V_{y^{-1}(x\vee y)}^*&\text{if $x\vee y<\infty$}\\
0&\text{if  $x\vee y=\infty$.}\end{cases}
\end{equation}
Nica showed that there is a $C^*$-algebra $C^*(G,P)$ which is generated by a universal Nica-covariant repesentation $w:P\to C^*(G,P)$, and we then have $C^*(G,P)=\clsp\{w_xw_y^*:x,y\in P\}$. Nica identified an amenability condition which implies that the representation $\pi_T:C^*(G,P)\to\TT(P)$ is faithful (see \cite{nica} or \cite[Theorem~3.7]{quasilat}). This amenability hypothesis is automatically satisfied when the group $G$ is amenable \cite[\S4.5]{nica}. Since the enveloping group $\qxqx$ of our semigroup $\nxnx$ is amenable, we can use Nica's theorem to view our Toeplitz algebra $\TT(\nxnx)$ as the $C^*$-algebra generated by a universal Nica-covariant representation $w:\nxnx\to \TT(\nxnx)$. (Here we use the lower case $w$ to denote the Toeplitz representation $T$ to emphasise that it has a universal property;  the obvious letter $t$ is not available because it will later denote time.) 

Nica studied $\TT(P)$ by exploiting what he called its ``crossed product-like structure,'' which the present authors subsequently recognised as that of a semigroup crossed product $B_P\rtimes P$ \cite{quasilat}. The underlying algebra $B_P$ is the closed subset of $\ell^\infty(P)$ spanned by the characteristic functions $1_x $ of the sets  $x P=\{y\in P:y\geq x\}$, which is a $C^*$-subalgebra because 
$1_x 1_y = 1_{ x \vee y}$. Nica showed that the homomorphisms from $B_P$ to $\C$ are given by the nonempty hereditary subsets $\omega$ of $P$ which are directed by the partial order on $G$: the corresponding homomorphism is defined by\[
\hat\omega(f):=\lim_{x\in \omega}f(x),
\]
which makes sense because $\omega$ is directed and $f$ is the uniform limit of functions which are eventually constant. Notice that $\hat\omega$ is characterised by the formula
\[
\hat \omega(1_x)=\begin{cases}1&\text{if $x\in \omega$}\\
0&\text{if $x\notin \omega$.}
\end{cases}
\]
The collection $\Omega$ of nonempty directed hereditary sets $\omega$, viewed as a subset of the compact space $\{0,1\}^P$, is now called the \emph{Nica spectrum} of $P$. 

An important tool in our analysis will be the realisation of $C^*(G,P)$ as a partial crossed product $C(\Omega_{\sn})\rtimes_\alpha (\qxqx)$ obtained in \cite[\S6]{topfree}. The space $\Omega_{\sn}$ used in \cite{topfree} looks slightly different: its elements are the hereditary subsets of $G$ which contain the identity $e$ of $G$. However, it was observed in \cite[Remark~6.5]{topfree} that $\omega\mapsto \omega\cap P$ is a homeomorphism of $\Omega_{\sn}$ onto the Nica spectrum  $\Omega$, so we can apply the general theories of \cite{topfree} and \cite{CL2} in our situation. For $x\in P$, the partial map $\theta_x$ has domain all of $\Omega$, and $\theta_x(\omega)$ is the hereditary closure $\her(x\omega)$ of $x\omega:=\{xy:y\in \omega\}$; since the domain of $\theta_g$ is empty unless $g\in PP^{-1}$, this completely determines $\theta$. The action $\alpha$ is defined by $\alpha_g(f)(\omega)=f(\theta_{g^{-1}}(\omega))$ when this makes sense, and $\alpha_g(f)(\omega)=0$ otherwise.

\subsection{Number theory} We write $\N$ for the semigroup of natural numbers (including $0$) under addition and $\nx$ for the semigroup of positive integers under multiplication. We also write $\Q^*$ for the group of non-zero rational numbers under multiplication, and $\Q^*_+$ for the subgroup of positive rational numbers. 

We denote the set of prime numbers by $\primes$, and we write $e_p(a)$ for the exponent of $p$ in the prime factorization of $a\in \nx$, so that $a=\prod_{p\in\primes}p^{e_p(a)}$; then $a\mapsto \{e_p(a)\}$ is an isomorphism of $\nx$ onto the direct sum $\bigoplus_{p\in\primes}\N$. We also use the \emph{supernatural numbers}, which are formal products $N=\prod_{p\in\primes}p^{e_p(N)}$ in which the exponents $e_p(N)$ belong to $\N\cup\{\infty\}$; thus the set $\sn$ of supernatural numbers is the direct product $\sn= \prod_{p \in\primes} p^{\N\cup \{\infty\}}$. For $M,N\in\sn$, we say that $M$ divides $N$ (written $M|N$) if $e_p(M)\leq e_p(N)$ for all $p$, and then any pair $M,N$ has a greatest common divisor $\gcd(M,N)$ in $\sn$ and a lowest common multiple $\lcm(M,N)$.

For $N\in \sn$, the \emph{$N$-adic integers} are the elements of the ring
\[
\Z/N := \varprojlim \big((\Z/a\Z):a\in\nx,\;a|N\big).
\] 
When $N$ is a positive integer, $\Z/N$ is the ring $\Z/N\Z$, and for each prime $p$, $\Z/p^\infty$ is the usual ring $\Z_p$ of $p$-adic integers. When $N=\nabla:=\prod_p  p^\infty$ is the largest supernatural number,  $\Z/\nabla$ is the ring $\hatz$ of integral ad{\`e}les. If $M$ and $N$ are supernatural numbers and $M|N$, then there is a canonical topological-ring homomorphism of $\Z/N$ onto $\Z/M$, and we write $r(M)$ for the image of $r\in \Z/N$ in~$\Z/M$.

It is standard practice to freely identify $\hatz$ with the product $\prod_p \Z_p$, and the next proposition gives a similar product decomposition for $\Z/N$. The main ingredient in the proof is the Chinese remainder theorem. 

\begin{proposition}\label{prodvsinvlim}
Let $N =\prod_{p} p^{e_p(N)} $ be a supernatural number. Then the map $r\mapsto \{r(p^{e_p(N)})\}_{p\in \primes}$ is a (topological ring) isomorphism
 of $\Z/N$ onto $\prod_{p\in\primes}\Z/p^{e_p(N)}$.
\end{proposition} 

Our arguments involve a fair bit of modular arithmetic, and we will often need to change base. So the next lemma will be useful.
\begin{lemma}\label{timesb}
Suppose that $a$ and $b$ are integers greater than $1$. Then the map $n\mapsto an$ induces a well-defined injection $\times a:\Z/b\Z\to \Z/ab\Z$; the image of this map is $\{m\in\Z/ab\Z:m\equiv 0\pmod a\}$, so we have a short exact sequence
\[
\xymatrix{
0\ar[r]&{\Z/b\Z}\ar[r]^{\times a} &\Z/ab\Z\ar[rr]^{\pmod a}&&\Z/a\Z\ar[r]&0.
}
\]
\end{lemma}

As a point of notation,  when $b$ is clear from the context, we write $an$ for the image of $n\in \Z/b\Z$ under the map $\times a$.

\section{A quasi-lattice order on $\qxqx$}\label{semigpisql}              

Let $\qxqx$ denote the semidirect product of the additive rationals by the 
multiplicative positive rationals, where the group operation and inverses are given by 
\begin{align*}
(r,x)(s,y)&= (r+xs, xy) \qquad\text{ for } r,s \in \Q \text{ and } x,y \in \qx,\ \text{ and }\\
(r,x)\inv&= (-x\inv r , x\inv) \qquad \text{ for } r \in \Q \text{ and } x \in \qx.
\end{align*}

\begin{proposition}
The group $\qxqx$ is generated by the elements $(1,1)$ and $\{(0,p): p\in \primes\}$ which satisfy  the relations
\begin{equation}\label{presentationqxqx}
(0,p) (1,1) = (1,1)^p (0,p) \quad \text{ and } \quad (0,p) (0,q) = (0,q) (0,p) \qquad \text{ for all } p,q\in \primes,
\end{equation}
and this is a presentation of $\qxqx$.
\end{proposition}

\begin{proof}
It is easy to check that the elements $(1,1)$ and $(0,p)$ satisfy \eqref{presentationqxqx}. Suppose $G$ is a group containing elements $s$ and $\{v_p:p\in \primes\}$ satisfying the relations $v_ps=s^pv_p$ and $v_pv_q=v_qv_p$. Since $\qx$ is the free abelian group generated by $\primes$ and $v_p$ commutes with $v_q$, the map $p\mapsto v_p$ extends to a homomorphism $v:\qx\to G$. Since $\Z$ is free abelian, for each $n\in \nx$ there is a homomorphism $\phi_n:n^{-1}\Z\to G$ satisfying $\phi_n(n^{-1}k)=
v_n^{-1}s^k v_n $, and these combine to give a well-defined homomorphism $\phi:\Q=\bigcup_nn^{-1}\Z$ into $G$. Now the first relation extends to $v_rs^k=s^{rk}v_r$, and it follows easily that $v$ and $\phi$ combine to give a homomorphism of the semidirect product $\qxqx$ into $G$. 
\end{proof}

We shall consider the unital subsemigroup $\nxnx$ of 
$\qxqx$, which has the same presentation, interpreted in the category of monoids. Since $(\nxnx)\cap(\nxnx)^{-1}=\{(0,1)\}$, the subsemigroup $\nxnx$ induces a left-invariant partial order on $\qxqx$ as follows: for  $(r,x)$ and  $(s,y)$ in
$\qxqx$, 
\begin{align}
(r,x)  \leq (s,y) 	& \iff (r,x)\inv (s,y) \in \nxnx  \nonumber \\
			& \iff (-x\inv r, x\inv) (s,y) \in \nxnx  \nonumber \\
			& \iff (-x\inv r + x\inv s , x\inv y) \in \nxnx \nonumber \\
			& \iff x\inv(s-r) \in \N \text{ \ and \ } x^{-1}y \in\nx. \label{lequseful}
\end{align}
Our first goal is to show that this ordering has the quasi-lattice property used in \cite{nica} and \cite{quasilat}. 

\begin{proposition}\label{qlproperty}
The pair $(\qxqx, \nxnx)$ is a quasi-lattice ordered group.
\end{proposition}
\begin{proof}
 By  \cite[Lemma 7]{CL1}, it suffices to show that if an element $(r,x)$ of $ \qxqx$ has an upper bound in $\nxnx$, then it has a least upper bound in $\nxnx$. Suppose $(k,c)\in \nxnx$ and $(r,x)\leq (k,c)$. Then from \eqref{lequseful} we have $k\in r+x\N$, so $(r+x\N)\cap \N$ is nonempty; let $m$ be the smallest element of $(r+x\N)\cap \N$. Write $x=a/b$ with $a,b\in\N$ and $a$, $b$ coprime. We claim that $(m,a)$ is a least upper bound for $(r,x)$ in $\nxnx$.
 
To see that $(r,x)\leq (m,a)$, note that $m\in r+x\N$ and $x^{-1}a=b\in\nx$, and apply \eqref{lequseful}. To see that $(m,a)$ is a \emph{least} upper bound, suppose that $(l,d)\in \nxnx$ satisfies $(r,x)\leq (l,d)$, so that by \eqref{lequseful} we have (i) $x^{-1}d\in \nx$ and (ii) $x^{-1}(l-r)\in\N$. Property (i) says that $a^{-1}bd=x^{-1}d$ belongs to $\nx$, which since $(a,b)=1$ implies that $a^{-1}d\in\nx$. Property (ii) implies that $l\in r+x\N$, so that $l\geq m:=\min((r+x\N)\cap \N)$. Property (ii) also implies that 
\[
a^{-1}b(l-m)=x^{-1}(l-m)\in x^{-1}((r+x\N)-(r+x\N))\subset \Z,
\]
which, since $(a,b)=1$, implies that $a^{-1}(l-m)\in \Z$. Since $l\geq m$, we have  $a^{-1}(l-m)\in \N$. Now we have $a^{-1}d\in \nx$ and $a^{-1}(l-m)\in\N$, which by \eqref{lequseful} say that $(m,a)\leq (l,d)$, as required.  
\end{proof}

\begin{remark}
Two elements $(m,a)$ and $(n,b)$ of $\nxnx$ have a common upper bound if and only if the set  $(m+a\N)\cap (n+b\N)$ is nonempty, and in fact it is easy to check that  \begin{equation} \label{Vcharacterisation}
(m,a) \vee (n,b) = \begin{cases} 
\infty  & \text{ if } (m+a\N)\cap (n+b\N) = \emptyset,       \\
 ( l  , \lcm(a,b)) & \text{ if } (m+a\N)\cap (n+b\N)\not = \emptyset,
\end{cases}
\end{equation}
where $l$ is the smallest element of $(m+a\N)\cap (n+b\N)$. In the next section we will see that $(m+a\N)\cap (n+b\N)\neq\emptyset$ if and only if $m-n$ is divisible by the greatest common divisor $\gcd(a,b)$, and provide an algorithm for computing $l$ when it exists. 
\end{remark}

As we remarked in the Introduction, we found Proposition~\ref{qlproperty} a little surprising, because the pair of semigroups $P=\nx$ and $Q=\N$ do not satisfy the hypotheses of \cite[Example~2.3.3]{nica}. Since its proof is really quite elementary, we stress that Proposition~\ref{qlproperty} has some important consequences. In particular, since the group $\qxqx$ is amenable, we can immediately deduce from the work of Nica discussed in \secref{Toeplitzalgs} that the Toeplitz algebra $\TT(\nxnx)$ enjoys a universal property.

\begin{corollary}\label{Toepl=univ}
The Toeplitz algebra $\TT(\nxnx)$ is generated by a universal Nica-covariant isometric representation $w:\nxnx\to \TT(\nxnx)$.
\end{corollary}

\section{The euclidean algorithm}\label{euclideanalgorithm}

Suppose that $c,d\in \N$ are relatively prime. Then we know from the usual euclidean algorithm that, for every $k\in\N$, there are integers $\alpha$ and $\beta$ such that $k=\alpha c-\beta d$. Since $\alpha+md$ and $\beta+mc$ then have the same property for each $m\in \Z$, we can further assume that $\alpha$ and $\beta$ are non-negative. Since the set 
\[
\{\alpha\in \N:\text{there exists $\beta\in \N$ such that $k=\alpha c-\beta d$\}}
\]
is bounded below, it has a smallest element, and then the corresponding $\beta$ is the smallest non-negative integer for which there exists an $\alpha$ with $k=\alpha c-\beta d$. Thus it makes sense to talk about the \emph{smallest non-negative solution} $(\alpha,\beta)$ of $k=\alpha c-\beta d$. In the proof of Theorem~\ref{toeplitzpresentation} we use the numbers $\alpha_i$ and $\beta_i$  arising in the following variation of the euclidean algorithm which computes this smallest solution $(\alpha,\beta)$. 

\begin{proposition}\label{Euclid}
Suppose $\gcd(c,d)=1$ and $k\in \N$. Define sequences $\alpha_n$, $\beta_n$ inductively as follows: 
\begin{itemize}
\item define $\alpha_0$ to be the unique non-negative integer such that $-c<k-\alpha_0c\leq 0;$
\item given $\alpha_i$ for $0\leq i\leq n$ and $\beta_i$ for $0\leq i<n$, define $\beta_n$ by
\begin{equation}\label{defbetan}
0\leq k-\Big(\sum_{i=0}^n\alpha_i\Big)c+\Big(\sum_{i=0}^n\beta_i\Big)d<d;
\end{equation}
\item given $\alpha_i$ for $0\leq i\leq n$ and $\beta_i$ for $0\leq i\leq n$, define $\alpha_{n+1}$ by
\begin{equation}\label{defalphan+1}
-c< k-\Big(\sum_{i=0}^{n+1}\alpha_i\Big)c+\Big(\sum_{i=0}^n\beta_i\Big)d\leq 0.
\end{equation}
\end{itemize}
Then there exist $n(\alpha)$ and $n(\beta)$ (which is either $n(\alpha)$ or $n(\alpha)-1$) such that $\alpha_i=0$ for $i>n(\alpha)$ and $\beta_i=0$ for $i>n(\beta)$, and the pair $(\alpha,\beta)=
\big(\sum_{i=0}^{n(\alpha)}\alpha_i,\sum_{i=0}^{n(\beta)}\beta_i\big)$
is the smallest non-negative solution of $k=\alpha c-\beta d$. 
\end{proposition} 

\begin{proof}
We know from the discussion at the start of the section that there is a smallest solution $(\alpha,\beta)$; our problem is to show that the sequences $\{\alpha_n\}$ and $\{\beta_n\}$ behave as described. We first note that if any $\alpha_n$ or any $\beta_n$ is zero, then so are all subsequent $\alpha_i$ and $\beta_i$. We deal with the cases $c>d$ and $c<d$ separately.

Suppose that $c>d$. Then for every $n\geq 0$, Equation~\eqref{defbetan} implies that
\[
-c\leq k-\Big(\sum_{i=0}^n\alpha_i\Big)c+\Big(\sum_{i=0}^n\beta_i\Big)d-c<d-c<0.
\]
so $\alpha_{n+1}$ is either $0$ (if we have equality at the left-hand end) or $1$. So the sequence $\{\alpha_n\}$ starts off $\{\alpha_0,1,1,\cdots\}$, and is eventually always $1$ or always $0$. Since $k=\alpha c-\beta d\leq \alpha c$, we have $\alpha_0\leq \alpha$. We define $n(\alpha)=\alpha-\alpha_0$, and claim that
$\alpha_n=1$ for $1\leq n\leq n(\alpha)$. To see this, suppose to the contrary that $\alpha_n=0$ for some $n$ satisfying $1\leq n\leq n(\alpha)$. Then \eqref{defbetan} and \eqref{defalphan+1} imply that
\[
k-\Big(\sum_{i=0}^{n-1}\alpha_i\Big)c+\Big(\sum_{i=0}^{n-1}\beta_i\Big)d= 0,
\]
which, since $\sum_{i=0}^{n-1}\alpha_i=\alpha_0+n-1<\alpha_0+(\alpha-\alpha_0)=\alpha$, contradicts the assumption that $(\alpha,\beta)$ is the smallest solution. So $\alpha_n=1$ for $1\leq n\leq n(\alpha)$, as claimed. But now $\alpha=\sum_{i=0}^{n(\alpha)}\alpha_i$, and \eqref{defbetan} becomes
\begin{equation}\label{slippery}
0\leq k-\Big(\sum_{i=0}^{n(\alpha)}\alpha_i\Big)c+\Big(\sum_{i=0}^{n(\alpha)}\beta_i\Big)d
=-\beta d+\Big(\sum_{i=0}^{n(\alpha)}\beta_i\Big)d<d,
\end{equation}
which is only possible if $-\beta +\sum_{i=0}^{n(\alpha)}\beta_i=0$; then we have equality in \eqref{slippery}, and this implies that $\alpha_n=0$ and $\beta_n=0$ for $n>n(\alpha)$. So when $c>d$, $n(\alpha)=\alpha-\alpha_0$ and either $n(\beta)=n(\alpha)-1$ (if $\beta_{n(\alpha)}=0$) or $n(\beta)=n(\alpha)$ (if $\beta_{n(\alpha)}\not=0$) have the required properties.

For $c<d$, a similar argument shows that $\beta_n=1$ for $0\leq n\leq \beta-1$, so $n(\beta):=\beta-1$ and either $n(\alpha)=n(\beta)$ or $n(\alpha)=n(\beta)+1$ have the required properties.
\end{proof}

If $k\in \Z$ and $k<0$, we can apply Proposition~\ref{Euclid} to $-k$ and the pair $d$, $c$, obtaining a smallest non-negative solution of $-k=\beta d-\alpha c$. Notice that  we then have $k=\alpha c-\beta d$. This situation occurs so often that it is worth making the following simplifying convention: 

 \begin{convention}\label{convsmallestsol}
 When $k$ is an integer and we say that ``$(\alpha,\beta)$ is the smallest non-negative solution of $k=\alpha c-\beta d$,'' we mean that $(\alpha,\beta)$ is the smallest non-negative solution of $k=\alpha c-\beta d$ when $k\geq 0$ (as before), and that $(\beta,\alpha)$ is the smallest non-negative solution of $-k=\beta d-\alpha c$ when $k<0$.
 \end{convention}
 
The next proposition explains why this discussion of the euclidean algorithm is relevant to the semigroup $\nxnx$. Recall that $\lcm(a,b)$ is the lowest common multiple of $a$ and $b$.
 
\begin{prop}\label{eucliduseful}
		Suppose that $(m,a)$ and $(n,b)$ are in $\nxnx$. Then $(m+a\N)\cap (n+b\N)$ is nonempty if and only if $\gcd(a,b)|m-n$. If so, write $a'=a/\gcd(a,b)$, $b'=b/\gcd(a,b)$, and let $(\alpha,\beta)$ be the smallest non-negative solution of $(n-m)/\gcd(a,b)=\alpha a'-\beta b'$   (using Convention~\ref{convsmallestsol}). Then $l:=m+a\alpha =n+b\beta$ is the smallest element of $(m+a\N)\cap (n+b\N)$, 
and we have 
\begin{align*}
(m,a) \vee (n,b) & = (l,\lcm(a,b))\\
(m,a)\inv (l,\lcm(a,b))&= (a^{-1}(l-m), a^{-1}\lcm(a,b)) = (\alpha, b'),\\
(n,b)\inv (l,\lcm(a,b))&= (b^{-1}(l-n), b^{-1}\lcm(a,b)) = (\beta, a').
\end{align*}\end{prop}

\begin{proof}
The discussion at the start of the section shows that 
\[
(m+a\N)\cap (n+b\N)\not=\emptyset\Longleftrightarrow (m+a\Z)\cap (n+b\Z)\not=\emptyset\Longleftrightarrow m\equiv n\pmod{\gcd(a,b)}.
\]
Then any solution of $(n-m)/\gcd(a,b)=\alpha a'-\beta b'$ will satisfy $m+a\alpha =n+b\beta$, and the smallest non-negative solution of $(n-m)/\gcd(a,b)=\alpha a'-\beta b'$ will give the smallest common value. The last two formulas are an easy calculation. 
\end{proof}

\section{A presentation for the Toeplitz algebra of $\nxnx$}\label{secpresentation}

Our goal in  this section is to verify the following presentation for $\TT(\nxnx)$. Recall from \secref{Toeplitzalgs} that $\TT(\nxnx)$ is generated by a universal Nica-covariant representation $w:\nxnx\to \TT(\nxnx)$.

\begin{theorem} \label{toeplitzpresentation} 
Let $A$ be the universal $C^*$-algebra generated by isometries $s$ and $\{v_p:p\in \primes\}$ satisfying the relations
\begin{itemize}
\item[] \begin{itemize}
\item[(T1)]\  $v_p s =  s^p v_p$,
\smallskip
\item[(T2)]\ $v_p v_q =  v_q v_p$,
\smallskip
\item[(T3)]\  $v_p^* v_q =  v_q v_p^*$  when $p\neq q$,
\smallskip
\item[(T4)]\ $s^* v_p = s^{p-1} v_p s^*$, and
\smallskip
\item[(T5)]\ $v_p^* s^k v_p = 0$ for $1 \leq k < p$.
\end{itemize}
\end{itemize}
Then there is an isomorphism $\pi$ of $\TT(\nxnx)$ onto $A$ such that $\pi(w_{(1,1)})=s$ and $\pi(w_{(0,p)})=v_p$ for every $p\in \primes$.
\end{theorem}

\begin{remark}
We usually use upper case $V$ or $W$ to denote isometric representations of semigroups, and lower case $v$ or $w$ if we are claiming that they have some universal property. Similarly, we usually write $S$ for a single isometry to remind us of the unilateral shift and $s$ for a single isometry with a universal property. We discovered towards the end of this project that our notation clashes with that used by Cuntz --- indeed, we couldn't have got it more different if we had tried. (He denotes his additive generator by $u$ and his isometric representation of $\nx$ by $s$.)  By the time we noticed this,
the shift $s$ and the isometries $v_p$ were firmly embedded in our manuscript and in our minds, and it seemed to be asking for trouble to try to correct them all, so we didn't. But we apologise for any confusion this causes.
\end{remark}

To prove this theorem, we show
\begin{itemize}
\item[(a)] that the operators $S=w_{(1,1)}$ and $V_p=w_{(0,p)}$ satisfy the relations (T1--5), and hence there is a homomorphism $\rho_w:A\to \TT(\nxnx)$ such that $\rho_w(s)=w_{(1,1)}$ and $\rho_w(v_p)=w_{(0,p)}$; and

\smallskip
\item[(b)] that the formula 
\[
X_{(m,a)}:=s^mv_a:=s^m{\textstyle \prod_{p\in\primes} v_p^{e_p(a)}}
\]
defines a Nica-covariant isometric representation $X=X^{s,v}$ of $\nxnx$ in $A$, and hence induces a homomorphism $\pi_{s,v}:\TT(\nxnx)\to A$.
\end{itemize}
Given these, it is easy to check that $\rho_w$ and $\pi_{s,v}$ are inverses of each other, and $\pi:=\pi_{s,v}$ is the required isomorphism.

In view of \eqref{genNicacov} and Proposition~\ref{eucliduseful}, an isometric representation $W$ of $\nxnx$ is Nica covariant if and only if
\begin{equation}
\label{covarianceformula}
W_{(m,a)}^* W_{(n,b)} = \begin{cases}
0 & \text{ if } m\not\equiv n\pmod{\gcd(a,b)} \\
W_{(\alpha,b')} W_{(\beta,a')}^* & \text{ if }   m\equiv n \pmod{\gcd(a,b)},
\end{cases}
\end{equation}
where $a' = a/\gcd(a,b)$, $b' = b/\gcd(a,b)$, and (using Convention~\ref{convsmallestsol}) $(\alpha,\beta)$ is the smallest non-negative solution of $(n-m)/\gcd(a,b)=\alpha a'-\beta b'$. The proof of Theorem~\ref{toeplitzpresentation} uses the euclidean algorithm of Proposition~\ref{Euclid} to recognise the $\alpha$ and $\beta$ appearing on the right-hand side of \eqref{covarianceformula}.

To prove (a) we note that (T1) holds because $(0,p)(1,1)=(p,1)(0,p)$ and $W$ is a homomorphism, and (T2) holds  because $(0,p)(0,q)=(0,pq)=(0,q)(0,p)$. Equations~(T3), (T4) and (T5) are the Nica covariance relation \eqref{covarianceformula} for $(m,a)=(0,p)$ and $(n,b)=(0,q)$; for $(m,a)=(1,1)$ and $(n,b)=(0,p)$; and for $(m,a)=(0,p)$ and $(n,b)=(k,p)$, respectively.

So now we turn to (b). The first observation, which will be used many times later, is that the relations (T1)--(T5) extend to the $v_a$, as follows.

\begin{lemma}\label{relsata}
Suppose that $s$ and $\{v_p:p\in \primes\}$ are isometries satisfying the relations \textnormal{(T1)--(T5)}. Then the isometries $v_a:=\prod_{p\in\primes} v_p^{e_p(a)}$ for $a\in \N^\times$ satisfy
\begin{itemize}
\item[] \begin{itemize}
\item[(T1')]\  $v_a s =  s^a v_a$,
\smallskip
\item[(T2')]\ $v_a v_b =   v_b v_a$,
\smallskip
\item[(T3')]\ $v_a^* v_b =  v_b v_a^*$  whenever $\gcd(a,b)=1$,
\smallskip
\item[(T4')]\ $s^* v_a = s^{a-1} v_a s^*$, and
\smallskip
\item[(T5')]\ $v_a^* s^k v_a = 0$ for $1 \leq k < a$.
\end{itemize}
\end{itemize}
\end{lemma}

After we have proved this lemma, a reference to (T5), for example, could refer to either the original (T5) or to (T5').

\begin{proof}
Equations (T1'), (T2') and (T3') follow immediately from their counterparts for $a$ prime. We prove (T4') by induction on the number of prime factors of $a$. We know from (T4) that (T4') holds when $a$ is prime. Suppose that (T4') is true for every $a\in \nx$ with $n$ prime factors, and $b=aq\in \nx$ has $n+1$ prime factors. Then
\begin{align*}\label{2paragraphsabove}
s^*v_b&=s^* v_{aq}  = s^* v_{a}v_{q} = s^{a-1}  v_a s^* v_q = s^{a-1}  v_a s^{q-1} v_q s^*\\
& = s^{a-1}   s^{a(q-1)} v_av_q s^*= s^{aq -1} v_{aq} s^*=s^{b-1}v_bs^*,
\end{align*}
and we have proved (T4'). For (T5'), first prove by induction on $n$ (using (T1) as well as (T4)) that $v_{p}^{*n}s^kv_{p}^n\not=0$ implies $p^n|k$. Then
\begin{align*}
v_a^*s^kv_a\not=0&\Longrightarrow v_p^{*e_p(a)}s^kv_p^{*e_p(a)}\not=0\ \text{ for all $p|a$}\\
&\Longrightarrow p^{e_p(a)}|k\ \text{ for all $p|a$}\\
&\Longrightarrow a|k,
\end{align*}
which is a reformulation of (T5').
\end{proof}

Relations (T1') and (T2') imply that 
$X$ is an isometric representation of $\nxnx$, and it remains for us to prove that the representation $X$ in (b) satisfies the Nica-covariance relation \eqref{covarianceformula}. Since
\begin{equation}\label{theproduct}
X_{(m,a)}^*X_{(n,b)}=(s^mv_a)^* s^n v_b =  v_a^* (s^*)^m s^n v_b,
\end{equation} 
the following lemma gives the required Nica-covariance (the formula 
\eqref{formofNica} which expresses this covariance in terms of generators will be useful later).

\begin{lemma}\label{covarianceonngenerators} 
Suppose that $s$ and $\{v_p:p\in \primes\}$ are isometries satisfying the relations \textnormal{(T1)--(T5)}.
For $m,n\in \N$ and $a,b\in \nx$,  we let $a':=a/\gcd(a,b)$, $b':=b/\gcd(a,b)$, and suppose that $(\alpha, \beta)$ is the smallest non-negative solution of $(n-m)/{\gcd(a,b)}=\alpha a'-\beta b'$. Then
\begin{equation}\label{formofNica}
v_a^*s^{*m}s^nv_b= \begin{cases}0 & \text{ if } m\not\equiv n\pmod{\gcd(a,b)} \\
s^{\alpha}v_{b'}v_{a'}^*s^{*\beta}& \text{ if }   m\equiv n \pmod{\gcd(a,b)}.
\end{cases}
\end{equation} 
\end{lemma}

\begin{proof}
First suppose that $m\not\equiv n\pmod{\gcd(a,b)}$, so that 
$(m,a) \vee (n,b)=\infty$. Then $\gcd(a,b)$ has a prime factor $p$ which does not divide $n-m$
 and we can write $n-m=cp+k$ with $0< k<p$. Now we factor $a=a_0p$, $b=b_0p$ and apply (T4') to get
\begin{align*}
v_a^*s^{*m}s^nv_b&=\begin{cases}
v_{a_0}^*v_p^*s^ks^{cp}v_pv_{b_0}&\text{if $c\geq 0$}\\
v_{a_0}^*v_p^*s^{*|c|p}s^kv_pv_{b_0}&\text{if $c<0$}
\end{cases}\\
&=\begin{cases}
v_{a_0}^*v_p^*s^kv_ps^{c}v_{b_0}&\text{if $c\geq 0$}\\
v_{a_0}^*s^{*|c|}v_p^*s^kv_pv_{b_0}&\text{if $c<0$};
\end{cases}
\end{align*}
in both cases, the inside factor $v_p^* s^k v_p$  vanishes by (T5), and we have $v_a^*s^{*m}s^nv_b = 0$, as required.

Suppose now that $m\equiv n \pmod{\gcd(a,b)}$, so that $(m,a) \vee (n,b)<\infty$. Write $k=(n-m)/\gcd(a,b)$. As in the proof of (T4'), we can use (T1') to pull $s^{k\gcd(a,b)}$ or $s^{*|k|\gcd(a,b)}$ past $v_{\gcd(a,b)}$ or $v_{\gcd(a,b)}^*$, obtaining
\[
v_a^*s^{*m}s^nv_b=
\begin{cases}
v_{a'}^*s^kv_{b'}&\text{if $k\geq 0$}\\
v_{a'}^*s^{*|k|}v_{b'}&\text{if $k<0$.}
\end{cases}
\]
It suffices by symmetry to compute $v_{a'}^*s^kv_{b'}$ for $k>0$.

Peeling one factor off $s^k$ and applying the adjoint of (T4') gives
\[
v_{a'}^* s^{k} v_{b'} =
sv_{a'}^* s^{*(a'-1)} s^{k-1 }v_{b'}^*=
\begin{cases}
sv_{a'}^*  s^{(k-a')}v_{b'}^*&\text{ if $k-a'>0$}\\
sv_{a'}^*  s^{*(a'-k)}v_{b'}^*&\text{ if $k-a'\leq 0$.}
\end{cases}
\]
If $k-a'>0$, we can peel another $s$ off $s^{k-a'}$, and pull it across $v_{a'}^*$; we can do it yet again if $k-2a'>0$. The number of times we can do this is precisely the number $\alpha_0$ appearing in the euclidean algorithm of  Proposition~\ref{Euclid}, applied to $a'$, $b'$ and $k$. We wind up with
\[
v_{a'}^* s^{k} v_{b'} =s^{\alpha_0}v_{a'}^*s^{*(\alpha_0 a'-k)}v_{b'}.
\]
Now we apply (T4') to pull factors of $s^*$ through $v_{b'}$: we can do this $\beta_0$ times, and obtain
\[
v_{a'}^* s^{k} v_{b'} =s^{\alpha_0}v_{a'}^*s^{(k-\alpha_0 a'+\beta_0 b')}v_{b'}s^{*\beta_0}.
\]
We can continue this process, using alternately the adjoint of (T4') to pull out factors  of $s$ to the left and (T4') to pull out $s^*$ to the right. This finishes when there aren't any left, and this is precisely when the euclidean algorithm terminates. Now the equations $\alpha=\sum_i\alpha_i$ and $\beta=\sum_j\beta_j$ from Proposition~\ref{Euclid} gives
\[
v_{a'}^* s^{k} v_{b'} =s^{\alpha_0}s^{\alpha_1}\cdots s^{\alpha_{n(\alpha)}}v_{a'}^*v_{b'}s^{*\beta_0}\cdots s^{*\beta_{n(\beta)}}=s^\alpha v_{a'}^*v_{b'}s^{*\beta}.
\]

Finally, we observe that since $a'$ and $b'$ are coprime,  (T3) implies that $v_{a'}^*v_{b'}=v_{b'}v_{a'}^*$. Thus, if $n-m>0$, we have $k>0$ and
\[
v_a^*s^{*m}s^nv_b =s^\alpha v_{a'}^*v_{b'}s^{*\beta}=s^\alpha v_{b'}v_{a'}^*s^{*\beta}.
\]
On the other hand, if $n-m<0$, we have
\begin{align*}
v_a^*s^{*m}s^nv_b &= v_{a'}^*s^{*|k|}v_{b'}= (v_{b'}^*s^{|k|}v_{a'})^*\\
& = (s^\beta v_{a'}v_{b'}^*s^{*\alpha})^*=s^\alpha v_{b'}v_{a'}^*s^{*\beta},
\end{align*}
where we now use Convention \ref{convsmallestsol} to interpret ``$(\alpha,\beta)$ is the smallest non-negative solution of $k=\alpha a' - \beta b'$''.
 \end{proof}

It follows from \lemref{covarianceonngenerators} that the representation $X$ is Nica covariant, and we have proved (b). This completes the proof of Theorem~\ref{toeplitzpresentation}.

\section{The Nica spectrum of $(\qxqx, \, \nxnx)$}\label{nicaspectrum}

To get a convenient parametrisation of the Nica spectrum, we need to identify the nonempty hereditary directed subsets of $\nxnx$. First we give some examples (which will turn out to cover all the possibilities).

\begin{proposition} \label{defAB}
Suppose $N$ is a supernatural number. For each $k\in \N$, we define
\[
A(k,N) : = \{  (m,a) \in \nxnx : a|N\text{ and }a^{-1}(k-m)\in \N\},
\]
and for each $\oldz \in \Z/N$, we recall that $r(a)$ denotes the projection of $r$ in $\Z/a$ and we define
\[
B(\oldz,N) : = \{ (m,a)  \in \nxnx : a | N \text{ and } m \in \oldz(a)\}. 
\]
Then $A(k,n)$ and $B(\oldz,N)$ are nonempty hereditary directed subsets of $\nxnx$. 
\end{proposition}

\begin{remark}
The map $(k,c)\mapsto A(k,c)=\{(m,a)\in \nxnx:(m,a)\leq (k,c)\}$ is the standard embedding of the quasi-lattice ordered semigroup $\nxnx$ in its spectrum, see \cite[Section 6.2]{nica}.
\end{remark}

\begin{proof}[Proof of Proposition~\ref{defAB}]
Let $N$ be given. 
If $(m,a)$ and $(n,b)$ are in $A(k,N)$, then  $(k,\lcm(a,b)) \in A(k,N)$ is a common upper bound for $(m,a)$ and $(n,b)$, and hence $A(k,N)$ is directed. To see that $A(k,N)$ is hereditary, suppose $(m,a) \in A(k,N)$ and $(0,1) \leq (n,b) \leq (m,a)$. Then $b|a$ and $b^{-1}(m-n)\in \N$. Since $a|N$ and $a^{-1}(k-m)\in \N$ we have $b | a | N$, and thus $b^{-1}(k-n)   = (b^{-1}a)a^{-1}(k-m)  + b^{-1}(m-n)$ belongs to $\N$.
Thus $A(k,N)$ is hereditary. 

We next prove that $B(r,N)$ is directed. Suppose $(m,a)$ and $(n,b)$ are in $B(\oldz,N)$, so that $a$ and $b$ divide $N$ and $m\in r(a)$, $n\in r(b)$. Since $r(a)=r(\lcm(a,b))(a)$, there exists $k\in\Z$ such that $m+ak\in r(\lcm(a,b))$, and similarly there exists $l\in\Z$ such that $n+bl\in r(\lcm(a,b))$; by adding multiples of $\lcm(a,b)$ to both sides, we can further suppose that $k,l\in \N$ and that $m+ak=n+bl=t$, say. Then $t\in (m+a\N ) \cap (n+b\N)$, and  $(t,\lcm(a,b))$ is an upper bound for $(m,a)$ and $(n,b)$. Since $t=m+ak\in r(\lcm(a,b))$, and $\lcm(a,b)$ divides $N$, this upper bound belongs to $B(\oldz,N)$. Thus $B(\oldz,N)$ is directed.

To see that $B(\oldz,N)$ is hereditary, suppose $(0,1)\leq (n,b)  \leq (m,a) \in B(\oldz,N)$. Then we have $b|a$ and $b^{-1}(m-n) \in \N$. Then $m\in r(a)\subset r(b)$, and since $n$ has the form $n=m-bk$ for some $k\in \N$, we have $n\in r(b)$ also. Thus $(n,b)\in B(\oldz,N)$, as required.
\end{proof}

\begin{lemma}\label{kandN} 
Suppose $\omega$ is a nonempty  hereditary directed subset of $ \nxnx $.
For each prime $p$ let $e_p(\omega) :=\sup\{e_p(a):(m,a)\in\omega\}\in \N\cup\{\infty\}$, and define a supernatural number by $N_{\omega}:=\prod_p p^{e_p(\omega)}$. Define  
$k_\omega \in \N \cup \{\infty\}$ by 
\[
k_\omega := \sup \{ m : (m,a) \in \omega \text{ for some } a\in \nx\}.
\]
Suppose $a | N_\omega$. Then there exists $m \in \N$ such that $(m,a) \in \omega$, and moreover
\begin{enumerate}
\item if $k_\omega < \infty$, then $(k_\omega,a) \in \omega$;
\smallskip
\item if $k_\omega = \infty$, then there is a sequence $n_i \in \N$ such that $(n_i, a) \in \omega$ and $n_i \to \infty$.
\end{enumerate}
\end{lemma}

\begin{proof}
For each prime $p$ with $e_p(a)>0$, we have $e_p(a)\leq e_p(\omega)$, so there exists $(m_p,b) \in \omega$ such that 
$e_p(a)\leq e_p(b)$. Then $(m_p,p^{e_p(a)}) \leq (m_p,p^{e_p(b)}) \leq (m_p,b)$, and $(m_p, p^{e_p(a)})$ belongs to $\omega$ because $\omega$ is hereditary. Since $\omega$ is directed, the finite set $\{(m_p,p^{e_p(a)}):p\in\primes,\;e_p(a)>0\}$ has an upper bound in $\omega$, and since $\omega$ is hereditary, $(m,c):= \vee_p (m_p, p^{e_p(a)})$ also belongs to $\omega$. But $c=\prod_p p^{e_p(a)}=a$, so we have found $m$ such that $(m,a)\in\omega$. 

When $k_\omega$ is finite, there exists $d\in\N^\times$ such that $(k_\omega,d) \in \omega$, and since $\omega$ is directed, it contains the element $(l,\lcm(a,d)) = (m,a) \vee (k_\omega,d)$, where $l := \min((m+a\N) \cap (k_\omega+d\N))$.
But then $l \leq k_\omega$ by definition of $k_\omega$, and since $l \in k_\omega+ d\N$, we conclude that $l = k_\omega$. Since 
$(k_\omega,a) \leq (k_\omega,\lcm(a,d))$, we deduce that $(k_\omega,a) \in \omega$, proving part (1).

To prove part (2), suppose $a|N_\omega$, and choose $(n_1,a) \in \omega$. Assume  that
we have obtained $n_1 < n_2 < \cdots  < n_i$
such that $(n_i,a) \in \omega$.
Since $\{m\in \N: (m,b) \in \omega\}$ is unbounded, we may choose $(m,b)$ with $m > n_i$.
Then  $(n_{i+1}, \lcm(a,b)):=(n_i,a) \vee (m,b) $ belongs to $\omega$; since  $(n_{i+1},a)\leq (n_{i+1},\lcm(a,b))$ and $\omega$ is hereditary, $(n_{i+1},a)$ belongs to $\omega$, and part (2) follows by induction.
\end{proof}

\begin{remark}\label{ABorder}
Part (2) of the lemma implies that $B(\oldz,M)$ is never contained in $A(k,N)$. The possible inclusions are characterized as follows:
\begin{align*}
B(t,N) \subset B(\oldz,M)\ &\iff\ N|M  \text{ and }
t(a) = \oldz(a) \text{ for every } a|N ;  
\\
 A(k,N) \subset B(\oldz,M)\ &\iff\  N|M\text{ and } k\in r(a)\text{ for every }a|N      ;\\
 A(l,N) \subset A(k,M) \ &\iff \   N|M \text { and }
k - l \in a\N \text{ for every } a | N. 
\end{align*}
For $N\in \sn \setminus \nx$, we have $k - l \in a\N$ for every  $a | N $ if and only if $k=l$,  so for such $N$, $A(l,N) \subset A(k,M)$ implies $k = l$.
Notice also that it follows easily from these inclusions that the sets $A(k,N)$ and $B(\oldz,N)$
are distinct for different values of the parameters. 
\end{remark}

Next we show that every hereditary directed subset of $\nxnx$ is either an $A(k,N)$ or a $B(r,N)$.

\begin{proposition}\label{spectrum}
Suppose $\omega$ is a nonempty hereditary directed subset of $\nxnx$, and let $k_\omega$
and $N_\omega$ be as in \lemref{kandN};
\begin{enumerate}
\item if $k_\omega< \infty$, then $\omega = A(k_\omega,N_\omega)$;

\smallskip
\item if $k_\omega = \infty$, then there exists
$r_{\omega}\in \Z/N_{\omega}$ such that $r_{\omega}(a)=m$ for every $(m,a)\in \omega$, and we then have $\omega=B(r_{\omega},N_{\omega})$. 
\end{enumerate}
\end{proposition}

\begin{proof}
Suppose first that $k_\omega < \infty$, and $(m,a) \in A(k_\omega,N_\omega)$. Then $a| N_\omega$ and 
 $a^{-1}(k_\omega-m) \in \N$. Then part (1) of \lemref{kandN} implies that
 $(k_\omega,a)$ is in $\omega$, and since $\omega$ is assumed to be directed and 
 $(m,a) \leq (k_\omega,a)$, we conclude that $(m,a)\in\omega$ and $A(k_\omega,N_\omega) \subset \omega$. On the other hand, suppose that $(m,a) \in \omega$. Since $(k_\omega,a) \in \omega$ and $\omega$ is directed, $(m,a) \vee (k_\omega,a)$ belongs to $\omega$; but $m\leq k_\omega$ by definition of $k_\omega$, so $\min ((m+a\N) \cap  (k_\omega+a\N)) = k_\omega$, $(m,a) \vee (k_\omega,a)  = (k_\omega,a) $, and $(m,a) \leq (k_\omega,a)$. Since $A(k_\omega,N_\omega)$ is hereditary,
 we conclude that $(m,a)$ is in $A(k_\omega,N_\omega)$, and $\omega \subset A(k_\omega,N_\omega)$.
 
Now suppose that $k_\omega =\infty$. We need to produce  a suitable $ \oldz_\omega \in \Z/N_\omega$. We know from \lemref{kandN} that for every $a|N_\omega$ there exists $(m,a) \in \omega$,  and
 we naturally want to take $\oldz_\omega(a)$ to be the class of $m$ in $\Z/a$. To see that this is well defined, suppose $(m,a)$ and $(n,a)$ are both in $\omega$; since $ \omega $ is directed, they have a common upper bound $(l,b)$, and then $(l-m) \equiv 0 \equiv (l-n) \pmod a$, so $m \equiv n \pmod a$.

Next we have to show that $\oldz_\omega  :=(\oldz_\omega(a))_{a|N_\omega}$ is an element of the inverse limit, or in other words that $a|b|N_\omega$ implies $\oldz_\omega(a)= \oldz_\omega(b)(a)$.
Let $m$ be such that $(m,b) \in \omega$, so that $m\in\oldz_\omega(b)$. Since $a|b$, we have $(m,a) \leq (m,b)$, and $(m,a)$ also belongs to $\omega$. Thus we also have $m \in\oldz_\omega(a)$, and
$\oldz_\omega(b)(a) =[m]=\oldz_\omega(a)$, as required. Thus there is a well-defined class $r_\omega$ in $\Z/N_\omega$ with the required property.

It is clear from the way we chose $\oldz_\omega$ that $\omega \subset B(\oldz_\omega,N_\omega)$, so it remains to show the reverse inclusion. Suppose $(m,a) \in B(\oldz_\omega,N_\omega)$. Since $a|N_\omega$ and $k_\omega = \infty$, part (2) of \lemref{kandN} implies that we can choose
$n>m$ such that  $(n,a) \in \omega$.  Now both $m$ and $n$ are in $\oldz_\omega(a)$, so $a|(n-m)$; since $n-m >0$, this implies that $a^{-1}(n-m)\in \N$, and we have $(m,a) \leq (n,a)$. Since $\omega$ is hereditary, $(m,a) \in \omega$. Thus $B(\oldz_\omega,N_\omega)\subset \omega$, and we have proved (2).
\end{proof}

\begin{corollary}\label{descOmega}
The Nica spectrum of $\nxnx$ is
\[
\Omega=\{A(k,M):M\in \sn,k\in\N\}\cup\{B(r,N):N\in \sn,r\in Z/N\}.
\]
\end{corollary}
 
To identify Cuntz's $\qn$ as the boundary quotient of $\TT(\nxnx)$, we need to identify the \emph{boundary} $\partial \Omega$ of $\Omega$, as defined in \cite[Definition 3.3]{purelinf} or \cite[Lemma~3.5]{CL2}.

\begin{proposition}\label{boundary}
Let $\nabla := \prod_p p^\infty$  be the largest supernatural number. Then the map $\oldz \mapsto B(\oldz, \nabla)$ is a homeomorphism of the finite integral adeles $\hatz$
onto the boundary $\partial \Omega$ of $(\qxqx,\nxnx)$. Under this homeomorphism, the left action of $(m,a)\in\nxnx$ on $\hatz$ is given, in terms of the ring operations in $\hatz$, by $(m,a)\cdot r=m+ar$.
\end{proposition}

A substantial part of the argument works in greater generality, and this generality will be useful for the construction of KMS states.

\begin{lemma}\label{omegaB} 
The subset $\Omega_B := \{B{(\oldz,N)} \in \Omega : N\in \sn , \ \oldz \in \Z/N\}$ is a closed subset of $\Omega$. For each fixed $N\in \sn$, the map $\oldz\mapsto B(\oldz,N)$ is a homeomorphism of $\Z/N$ onto a closed subset of $\Omega_B$.
\end{lemma}

\begin{proof}
Suppose that $B(r_\lambda,N_\lambda)\to \omega$ in $\Omega$, so that
\[
B(r_\lambda,N_\lambda)^{^{\wedge}}(1_{m,a})\to \hat\omega(1_{m,a})\ \text{ for every $(m,a)\in\nxnx$.}
\]
Since the sets in $\Omega$ are non-empty, there exists $(m,a)$ such that $\hat\omega(1_{m,a})=1$. Then there exists $\lambda_0$ such that
\begin{align*}
\lambda\geq \lambda_0
&\Longrightarrow B(r_\lambda,N_\lambda)^{^{\wedge}}(1_{m,a})=1\\
&\Longrightarrow a|N_\lambda\text{ and } m\in r_\lambda(a)\\
&\Longrightarrow B(r_\lambda,N_\lambda)^{^{\wedge}}(1_{m+ka,a})=1\text{ for all $k\in \N$.}
\end{align*}
But this implies that the integer $k_{\omega}$ in Lemma~\ref{kandN} is infinity, and Proposition~\ref{spectrum}(2) implies that $\omega=B(\oldz_\omega,N_\omega)$. Thus $\Omega_B$ is closed. 

Since $\Z/N$ is compact and $r\mapsto B(r,N)$ is injective (see Remark~\ref{ABorder}), it suffices to prove that $r\mapsto B(r,N)$ is continuous. So suppose that $r_\lambda\to r$ in $\Z/N$, and let $a\in \nx$; we need to show that
\begin{equation}\label{togetctuity}
B(r_\lambda,N)^{^{\wedge}}(1_{m,a})\to B(r,N)^{^{\wedge}}(1_{m,a})\ \text{ for every $m\in \N$.}
\end{equation}
If $a\notdiv N$, then $B(r_\lambda,N)^{^{\wedge}}(1_{m,a})=0=B(r,N)^{^{\wedge}}(1_{m,a})$. So suppose $a|N$. Then
\[
B(r,N)^{^{\wedge}}(1_{m,a})
=\begin{cases}
1&\text{ if $m\in r(a)$,}\\
0&\text{ otherwise.}
\end{cases}
\]
Since the maps $r\mapsto r(a)$ are continuous, we can choose $\lambda_0$ such that $\lambda\geq \lambda_0\Longrightarrow r_\lambda(a)=r(a)$. But then $m\in r_{\lambda}(a)$ if and only if $m\in r(a)$, and 
\[
\lambda\geq \lambda_0\Longrightarrow
B(r_\lambda,N)^{^{\wedge}}(1_{m,a})= B(r,N)^{^{\wedge}}(1_{m,a}),
\] 
confirming \eqref{togetctuity}.
\end{proof}

\begin{proof}[Proof of Proposition~\ref{boundary}]
By definition, $\partial \Omega$ is the closure in the Nica spectrum $\Omega$ of the set of maximal hereditary directed subsets (see \cite[Definition 3.3]{purelinf} or \cite[Lemma~3.5]{CL2}). From \proref{spectrum} and the characterization of the inclusions given in
 \remref{ABorder}, we see that a hereditary directed set is maximal if and only if it has the form  $B(\oldz,\nabla)$. Since $\{B(\oldz,\nabla): \oldz \in \hatz\}$ is the image of the compact space $\hatz$ under the homeomorphism
 $\oldz \mapsto B(\oldz,\nabla)$ from \lemref{omegaB}, it is already closed and is equal to $\partial\Omega$.
 
The action $\theta_{(m,a)}$ on $\Omega$ satisfies 
\begin{align}\label{calcaction}
\theta_{(m,a)}(B(\oldz,\nabla)) & = \her((m,a)B(r,\nabla))\\
&=\her \{(m,a)(n,b): n\in r(b)\} \notag\\
& = \her  \{(m+an,ab):n\in r(b)\} \notag\\
& =\her \{(k, ab) : k \in (m+ar)(ab)\notag\},
\end{align}
since $n\in r(b)\Longleftrightarrow an\in (ar)(ab)$. But this is precisely $\{(k,c):k\in (m+ar)(c)\}=B(m+ar\nabla)$.
\end{proof}

\begin{remark}  
We think of supernatural numbers as limits of (multiplicatively) increasing sequences in $\N^\times$, and of classes in $\Z/N$ as limits of (additively) increasing sequences in $\N$. So the set $\Omega_B$ lies ``at additive infinity" and we call it the {\em additive boundary} of $\Omega$. 
The set $\Omega_A:=\{A(k,N):N\notin\nx\}$ lies ``at multiplicative infinity'',  and we call it the {\em multiplicative boundary}. Each of these defines a natural quotient of $\TT(\nxnx)$, and we plan to discuss these quotients elsewhere. The minimal boundary $\partial \Omega$ characterised in 
\proref{boundary}  lies at both additive and multiplicative infinity, and might be more descriptively called
the {\em affine boundary} of $\Omega$. 
 \end{remark}

In our construction  of KMS states in \secref{secconstructionKMSground} we need a product decomposition of the additive boundary $\Omega_B$ over the set $\primes$. 
We describe the factors in the next Lemma, and the product decomposition in the following Proposition.

\begin{lemma}
For each prime $p$, the set 
\[
X_p := \{B(\oldz,p^k): k \in \N \cup\{\infty\}, \ \oldz \in \Z/{p^k}\},
\]
is a closed subset of $\Omega$, and each singleton set $\{B(r,p^k)\}$ with $k<\infty$ is an open subset of $X_p$. 
\end{lemma}

\begin{proof}
Suppose that the net $\{ B(r_{\lambda},p^{k_\lambda}) : \lambda \in \Lambda\}$ converges in  $\Omega$; since $\Omega_B$ is closed in $\Omega$, the limit has the form  $ B(r,N)$, and it suffices to prove that $B(r,N)\in X_p$, or, equivalently, that $N=p^k$ for some $k\in \N\cup \{\infty\}$. Suppose that $a|N$. Then $(r(a),a)\in B(r,N)$, so $B(r,N)^{^{\wedge}}(1_{r(a),a})=1$, and there exists $\lambda_0$ such that
\begin{equation}\label{N=p^k}
\lambda\geq \lambda_0\Longrightarrow B(r_\lambda,p^{k_\lambda})^{^{\wedge}}(1_{r(a),a})=1\Longrightarrow (r(a),a)\in B(r_\lambda,p^{k_\lambda})
\Longrightarrow a|p^{k_\lambda}.
\end{equation}
Since every divisor of $N$ is a power of $p$, so is $N$. Thus $X_p$ is closed.

Now suppose that $k<\infty$, and $r\in \Z/p^k$. To see that $\{B(r,p^k)\}$ is open, it suffices to prove that if $B(r_{\lambda},p^{k_\lambda})\to B(r,p^k)$, then $B(r_{\lambda},p^{k_\lambda})$ is eventually equal to $B(r,p^k)$ (for then the complement $X_p\setminus \{B(r,p^k)\}$ is closed). Choose an integer $n$ in the class $r$. Then the element $(n,p^k)$ of $\nxnx$ belongs to $B(r,p^k)$, so the argument in \eqref{N=p^k} implies that there exists $\lambda_1$ such that 
\begin{align}\label{proplambda1}
\lambda\geq \lambda_1&\Longrightarrow(n,p^k)\in B(r_{\lambda},p^{k_\lambda})\\
&\Longrightarrow p^k|p^{k_\lambda}\text{ and }n\in r_\lambda(p^{k})\notag\\
&\Longrightarrow p^k|p^{k_\lambda}\text{ and }r=r_\lambda(p^{k});\notag
\end{align} 
in particular, we have $k\leq k_\lambda$ for $\lambda\geq \lambda_1$. On the other hand, no element of the form $(m,p^{k+1})$ belongs to $B(r,p^k)$, so $B(r,p^k)^{^{\wedge}}(1_{m,p^{k+1}})=0$, and there exists $\lambda_2$ such that
\begin{equation}\label{negmembership}
\lambda \geq \lambda_2\Longrightarrow B(r_\lambda,p^{k_\lambda})^{^{\wedge}}(1_{m,p^{k+1}})=0\Longrightarrow (m,p^{k+1})\notin B(r_\lambda,p^{k_\lambda})\text{ for $0\leq m<p^{k+1}$;}
\end{equation}
since membership of an element $(m,a)$ in a set $B(t,M)$ depends only on the class of $m$ in $\Z/a$, at least one $m$ in the range belongs to $r_\lambda(p^{k+1})$, so we deduce from \eqref{negmembership} that
\[
\lambda \geq \lambda_2\Longrightarrow p^{k+1}\notdiv p^{k_\lambda}\Longrightarrow k_\lambda\leq k.
\]
Now we choose $\lambda_3$ such that $\lambda_3\geq \lambda_1$ and $\lambda_3\geq \lambda_2$, and then $k_\lambda=k$ for $\lambda\geq \lambda_3$. Since \eqref{proplambda1} says that $r_\lambda(p^{k})=r$ for $\lambda\geq \lambda_3\geq\lambda_1$, we eventually have $r_\lambda=r_\lambda(p^{k_\lambda})=r_\lambda(p^k)=r$, and hence 
\[
\lambda\geq \lambda_3\Longrightarrow B(r_\lambda,p^{k_\lambda})=B(r,p^k), 
\]
as required.
\end{proof}

\begin{prop}\label{prodstructure} The map $f:B(\oldz,N)\mapsto \{B(\oldz(p^{e_p(N)}), p^{e_p(N)}):p\in\primes\}$ is a homeomorphism of the additive boundary $\Omega_B$ onto the product space $\prod_{p\in \primes} X_p$.
\end{prop}

\begin{proof}
For $p\in \primes$ we define $f_p:\Omega_B\to X_p$ by $f_p(B(r,N))=B(r(p^{e_p(N)}), p^{e_p(N)})$. Then the maps $f_p$ are the coordinate maps of $f$, and $f$ is continuous if and only if all the $f_p$ are. So we fix $p$, and consider a convergent net $B(r_{\lambda},N_{\lambda})\to B(r,N)$ in $\Omega_B$. Let $(m,a)\in \nxnx$. Then eventually
\begin{equation}\label{given}
B(r_\lambda,N_\lambda)^{^{\wedge}}(1_{m,a})= B(r,N)^{^{\wedge}}(1_{m,a}),
\end{equation}
and we need to show that we eventually have
\begin{equation}\label{needed}
B(r_\lambda(p^{e_p(N_\lambda)}), p^{e_p(N_\lambda)})^{^{\wedge}}(1_{m,a})= B(r(p^{e_p(N)}), p^{e_p(N)})^{^{\wedge}}(1_{m,a}).
\end{equation}
Both sides of \eqref{needed} vanish unless $a=p^k$, so we just need to consider $a=p^k$. But then for all $B(t,M)$ we have
\[
B(t,M)^{^{\wedge}}(1_{m,p^k})=B(t(p^{e_p(M)}), p^{e_p(M)})^{^{\wedge}}(1_{m,p^k}),
\]
so for $a=p^k$, \eqref{needed} follows immediately from \eqref{given}. Thus $f_p$ is continuous, and so is $f$.

To see that $f$ is injective, we suppose $f(B(r,N))=f(B(s,M))$. Then
\begin{align*}
f(B(r,N))=f(B(t,M))
&\Longrightarrow f_p(B(r,N))=f_p(B(t,M))\text{ for all $p\in \primes$}\\
&\Longrightarrow B(r(p^{e_p(N)}), p^{e_p(N)})=B(t(p^{e_p(M)}), p^{e_p(M)})\text{ for all $p\in \primes$}\\
&\Longrightarrow p^{e_p(N)}=p^{e_p(M)}\text{ and }r(p^{e_p(N)})=t(p^{e_p(N)})\text{ for all $p\in \primes$}\\
&\Longrightarrow N=M\text{ and }r(a)=t(a)\text{ for all $a$ such that $a|N$}\\
&\Longrightarrow N=M\text{ and }r=t\text{ in $\Z/N=\textstyle{\varprojlim_{a|N}}\Z/a\Z$.}
\end{align*}
To see that $f$ is surjective, suppose that $\{B(r_p,p^{k_p}):p\in \primes\}$ is an element of $\prod_p X_p$. Take $N$ to be the supernatural number $\prod_p p^{k_p}$. Since the map $r\mapsto \{r(p^{k_p}):p\in \primes\}$ is a homeomorphism of $\Z/N$  onto $\prod_{p\in\primes}\Z/p^{e_p(N)}$ (by \proref{prodvsinvlim}), there exists $r\in \Z/N$ such that $r(p^{k_p})=r_p$ for all primes $p$. Then $\{B(r_p,p^{k_p})\}=f(B(r,N))$, and $f$ is onto.

We have now shown that $f$ is a bijective continuous map of the compact space $\Omega_B$ onto $\prod_p X_p$, and hence $f$ is a homeomorphism. 
\end{proof}

\section{Cuntz's $\qn$ as a boundary quotient}\label{sectionqn}
The $C^*$-algebra considered by Cuntz  in \cite{cun2} is the universal $C^*$-algebra $\qn$ generated by a unitary $s$ and isometries $\{u_a:a\in \nx\}$ satisfying
\begin{itemize}
\item[] \begin{itemize}
\item[(C1)]\  $u_a s = s^a u_a$ for $a\in \nx$,
\smallskip
\item[(C2)]\ $u_au_b=u_{ab}$ for $a,b\in\nx$, and
\smallskip
\item[(C3)]\ $\sum_{k=0}^{a-1}  s^k u_au_a^*s^{*k}=  1$ for $a\in\nx$.
\end{itemize}
\end{itemize}
We aim to prove that $\qn$ is the boundary quotient of the Toeplitz algebra $\TT(\nxnx)$, and it is helpful for this purpose to have a slightly different presentation of $\qn$ which looks more like the presentation of $\TT(\nxnx)$ in \thmref{toeplitzpresentation}.

\begin{prop}\label{presentqn}
$\qn$ is the universal $C^*$-algebra generated by isometries $s$ and $\{v_p: p\in \primes\}$ satisfying 
\begin{itemize}
\item[] \begin{itemize}
\item[(Q1)]\  $v_p s = s^p v_p$ for every $p\in\primes$,
\smallskip
\item[(Q2)]\ $v_p v_q = v_q v_p$ for every $p,q\in\primes$,
\smallskip
\item[(Q5)]\ $\sum_{k=0}^{p-1}  (s^k v_p) (s^k v_p)^*=  1$ for every $p\in\primes$, and
\smallskip
\item[(Q6)]\ $ss^*=1$.
\end{itemize}
\end{itemize}
We then also have
\begin{itemize}
\item[] \begin{itemize}
\item[(Q3)]\  $v_p^*v_q = v_qv_p^*$ for $p,q\in\primes$ and $p\not=q$, and
\smallskip
\item[(Q4)]\ $s^*v_p=s^{p-1}v_ps^*$ for every $p\in\primes$.
\end{itemize}
\end{itemize}
\end{prop}

\begin{proof}
If $s$ is unitary and $u_a$ satisfy (C1), (C2) and (C3), then clearly $s$ and $v_p:=u_p$ satisfy (Q1), (Q2), (Q5) and (Q6). Suppose, on the other hand, that $s$ and $v_p$ satisfy (Q1), (Q2), (Q5) and (Q6), and define $u_a:=\prod_{p\in\primes} v_p^{e_p(a)}$. Then (Q1) implies that $v_ps^k=s^{kp}v_p$; thus $v_p^ns=s^{p^n}v_p^n$, and it follows that $v_as=s^av_a$ for all $a$, which is (C1). Equation (Q2) implies that the $v_p$ form a commuting family, and (C2) follows easily. To prove (C3), it suffices to show that if (C3) holds for $a=b$ and $a=c$, then it holds also for $a=bc$. So suppose (C3) holds for $a=b$ and $a=c$, and note that 
\[
\{k:0\leq k<bc\}=\{l+mb:  0\leq l<b,\ 0\leq m<c  \}.
\]
Thus, using (C1) and (C2), we have
\begin{align*}
\sum_{k=0}^{bc-1}  s^k u_{bc}u_{bc}^*s^k
&=\sum_{l=0}^{b-1}\sum_{m=0}^{c-1}s^ls^{mb}u_bu_cu_c^*u_b^*s^{*mb}s^{*l}\\
&=\sum_{l=0}^{b-1}s^lu_b\Big(\sum_{m=0}^{c-1}s^mu_cu_c^*s^{*m}\Big)u_b^*s^{*l},
\end{align*}
which equals $1$ because (C3) holds for $a=c$ and $a=b$. Thus $\{s,u_a\}$ satisfies (C1)--(C3), and the two presentations are equivalent.

Since $s$ is unitary, multiplying (Q1) on the left and right by $s^*$ gives (Q4). To see (Q3), we apply (C3) with $a=pq$ and 
\begin{equation}\label{Cuntz*commute}
v_p^*v_q=v_p^*\Big(\sum_{k=0}^{pq-1} s^ku_{pq}u_{pq}^*s^{*k}\Big)v_q
=\sum_{k=0}^{pq-1}v_p^*s^kv_pv_qv_p^*v_q^*s^{*k}v_q,
\end{equation}
where we used that $u_{pq}=v_pv_q=v_qv_p$. Since $v_p^*s^kv_p=0$ unless $p|k$, and $v_q^*s^kv_q=0$ unless $q|k$, the only non-zero term in the sum on the right of \eqref{Cuntz*commute} occurs when $k=0$, and we have
\[
v_p^*v_q=v_p^*v_pv_qv_p^*v_q^*v_q=v_q  v_p^*  .
\]
\end{proof}
Clearly condition (Q3) implies that $v_m^*v_n = v_nv_m^*$ for $m,n \in \nx$ with $\gcd(m,n)=1$
(this is \cite[Lemma 3.2(c)]{cun2} and has already been observed as (T3) $\implies$ (T3') in Lemma 4.3.).

\begin{corollary}\label{qnquotient}
Cuntz's $C^*$-algebra $\qn$ 
is the quotient of the Toeplitz algebra $\TT(\nxnx)=  C^*(s,v_p : p\in\primes  )  $ by the ideal $I $ generated by the elements $1-ss^*$ and $\{1-\sum_{k=0}^{p-1}  (s^k v_p) (s^k v_p)^*:p\in\primes\}$.
\end{corollary}

\begin{proof} 
Relations (Q1) and (Q2) are the same as (T1) and (T2), and hence hold in any quotient of $\TT(\nxnx)$; clearly (Q5) and (Q6) hold in $\TT(\nxnx) / I$. 
So \proref{presentqn} gives a homomorphism $\pi: \qn \to \TT(\nxnx) / I$.

On the other hand relations  (T1--4) are the same as (Q1--4), and hence hold in $\qn$; (Q5) implies that the isometries $\{s^k v_p: 0 \leq k <p\}$ have mutually orthogonal ranges, which is the content of (T5). So \thmref{toeplitzpresentation}
 gives a homomorphism $\rho: \TT(\nxnx) \to \qn$ that vanishes  on $I$, and hence induces a 
homomorphism $\tilde{\rho} :  \TT(\nxnx) / I \to \qn $  which is an inverse for $\pi$. 
\end{proof}

Recall from \cite[Lemma 3.5]{CL2}
that the boundary $\partial \Omega$ of a quasi-lattice order $(G,P)$
is the spectrum  (in the sense of \cite[Definition 4.2]{topfree}) of 
the elementary relations $\prod_{x\in F} (1 - W_x W_x^*) = 0$
corresponding to the sets in the family 
\begin{equation}\label{collectionF}
\mathcal F:= \{ F\subset P: |F|<\infty\text{ and }\forall\ y\in P  \  \exists\ x\in F \text{ such that } x\vee y <\infty\}
\end{equation}
from  \cite[Definition 3.4]{CL2},
taken together with the Nica relations from \cite[Proposition 6.1]{topfree}.
 Since we are working with covariant isometric representations, we will carry the implicit assumption that the Nica relations  always hold,
so the spectrum of a set  $\mathcal R$ of extra relations
is always a subset, denoted $\Omega(\mathcal R)$, of the Nica spectrum $\Omega$. In this notation, \cite[Lemma 3.5]{CL2} says that $\partial\Omega=\Omega(\mathcal{F})$, and the set $\Omega_\mathcal{N}=\Omega(\emptyset)$ of \cite[\S6]{topfree} is just $\Omega$.

Since $(\TT(\nxnx),w)$ is universal for Nica-covariant representations, Theorem~6.4 of \cite{topfree} implies that $\TT(\nxnx)$ is canonically isomorphic to the partial crossed product $C(\Omega)\rtimes (\qxqx)$. The \emph{boundary quotient} of \cite{CL2} is then the partial crossed product $C(\partial \Omega) \rtimes (\qxqx)$, which by  \cite[Theorem 4.4]{topfree} and \cite[Proposition 6.1]{topfree} is isomorphic to the universal $C^*$-algebra generated by
a Nica-covariant semigroup of isometries $W$ subject to the 
extra (boundary) relations 
\begin{equation*}
\prod_{x\in F} (1 - W_x W_x^*) = 0 \quad \text{ for } F\in \mathcal F.
\end{equation*}

\begin{theorem}\label{qnisboundaryquotient}
Cuntz's $C^*$-algebra $\qn$ is the boundary quotient $C(\partial \Omega) \rtimes (\qxqx)$ of the Toeplitz algebra $\TT(\nxnx)$.
\end{theorem}

\begin{proof}
Since $(\TT(\nxnx),w)$ is universal for Nica-covariant isometric representations of $\nxnx$, Corollary~\ref{qnquotient} implies that $\qn$ is universal for Nica-covariant representations $(S,V)$ of $\nxnx$ which satisfy
\begin{equation}\label{extrarelsinqn}
1-SS^*=0\ \text{ and }\ 1-\sum_{k=0}^{p-1} (S^kV_p)(S^kV_p)^*=0\ \text{ for $p\in \primes$.}
\end{equation}
Since the terms in the sum are mutually orthogonal projections, the relations \eqref{extrarelsinqn} are equivalent to 
\begin{gather}
1-SS^*=0,\ \text{ and }\label{q4}\\
\prod_{k=0}^{p-1} \big(1- (S^k V_p) (S^k V_p)^*\big)=  0\ \text{ for every $p\in \primes$.}\label{q5}
\end{gather}
We will prove that $\partial\Omega:=\Omega(\mathcal F)$ coincides with $\Omega(\{\eqref{q4},\eqref{q5}\})$.

To see that $\partial\Omega\subset\Omega(\{\eqref{q4},\eqref{q5}\})$, it suffices to show that $\{(1,1)\}$ and $\{(k,p) : 0 \leq k < p\}$ belong to $\mathcal F$.
Suppose $(m,a) \in \nxnx$. Since $(1,1) \vee (m,a) = (m,a)$ when $m > 0$ and
$(1,1) \vee (0,a) = (a,a)$, the set  $\{(1,1) \}$ is in $ \mathcal F$.
On the other hand, $m$ is in exactly one coset modulo $p$, say $m \in k+p\N$,  
and then $(m+a\N) \cap (k+p\N) \neq \emptyset$, so $(m,a) \vee (k,p) < \infty$.
Thus $\{(k,p) : 0 \leq k < p\}$ is in $\mathcal F$. 
Hence $\partial\Omega\subset\Omega(\{\eqref{q4},\eqref{q5}\})$.

For the reverse inclusion, we use the parametrization of the spectrum obtained in \lemref{spectrum}. 
Suppose $\omega\in \Omega(\{\eqref{q4},\eqref{q5}\})$. Then, since  
the spectrum of a set of relations is invariant by Proposition~4.1 of \cite{topfree},  $\omega$ is a hereditary directed subset of $\nxnx$ 
such that for all $(m,a) \in \nxnx$, the $(m,a)$-translates of the relations \eqref{q4} and \eqref{q5} (corresponding to conjugation of the relations by the isometry corresponding to $(m,a)$) hold at the point $\omega$. Thus
\begin{gather}\label{oldqp4ma}
 \hat{\omega}(1_{(m,a)} - 1_{(m+1,a)})= 0 \quad \text{ for all } (m,a) \in \nxnx,\text{ and}\\
\label{oldqp5ma}
 \displaystyle \hat{\omega}\Big(\prod_{k = 0}^{p-1} (1_{(m,a)} - 1_{(m+ak ,ap)}) \Big)  = 0 \quad \text{ for all } (m,a) \in \nxnx.
\end{gather}
From \eqref{oldqp4ma} we see that if $(m,a) \in \omega$, then $(m+1,a) \in \omega$; none of the
$A(k,n)$ have this property (take $m = k$), we have $\omega=B(\oldz,N)$ for some $N \in \sn$ and $\oldz \in \Z/N$. Now, using \eqref{oldqp5ma}, we get 
\begin{equation}\label{oldqp4z} 
\prod_{k = 0}^{p-1} \Big(B(\oldz,N)^{^{\wedge}} (1_{(m,a)}) -  B(\oldz,N)^{^{\wedge}} (1_{(m+ak ,ap)} )\Big)  = 
 0\quad \text{ for all } (m,a) \in \nxnx,
\end{equation}
Suppose now that $a | N$ and $p$ is a prime. Then for every $m\in \oldz(a)$ we have $(m,a) \in B(\oldz,N)$, and \eqref{oldqp4z} implies that there 
exists $k $ such that $(m+ak, ap) \in B(\oldz,N)$, which implies in particular that $ap|N$. Thus
 $N$ is the largest supernatural number $\nabla$,
 and $A = B(\oldz,\nabla) \in \partial \Omega$
by \proref{boundary}. Thus $\partial \Omega\supset \Omega(\{\eqref{q4},\eqref{q5}\})$ and we have proved that $\Omega(\mathcal F) = \Omega(\{\eqref{q4},\eqref{q5}\})$, as required.
\end{proof}

The \emph{core} of a quasi-lattice ordered group $(G, P)$ is the subgroup $G_0$ of $G$ generated by the monoid 
\[
P_0 = \{ x\in P: x\vee y <\infty \text{ for all } y\in P\}
\]
(see \cite[Definition 5.4]{CL2}). By \cite[Proposition 5.5]{CL2}, the partial action of $G$ on $\Omega$ is topologically free if and only if its restriction to the core $G_0$ is topologically free. So we want to identify the core:

\begin{lemma}
The core  of $(\qxqx, \nxnx)$ is $(\Z \rtimes \{1\} ,\N \rtimes \{1\})$.
\end{lemma}

\begin{proof}
Each $(m,1)$ is in the core, 
because $k := \min (m+\N) \cap (n+b\N)$ is always finite and by \eqref{Vcharacterisation}, we have $(m,1) \vee (n,b) = (k,b) \in \nxnx$.
Suppose now  $a \neq 1$; then $m \neq m+1 \pmod a$, so $(m,a) \vee ((m+1), a) = \infty$ by  \eqref{Vcharacterisation},
and thus  $(m,a)\notin P_0$.
\end{proof}

\begin{proposition}\label{paamentf}
The partial action of $\qxqx$ on the boundary $\partial \Omega$ is amenable and 
topologically free.
\end{proposition}

\begin{proof}
The expectation of $C_0(\partial\Omega)\rtimes(\qxqx)$ onto $C_0(\partial\Omega)$ is obtained by averaging over the dual coaction of $\qxqx$, and hence is faithful (by the argument of \cite[Lemma~6.5]{quasilat}, for example).
Thus the  partial action of $\qxqx$ on $\partial\Omega$ is amenable. Next, recall from Proposition~\ref{boundary} that $(k,1)B(r,\nabla)=B(r+k,\nabla)$; since $B(\oldz +k, \nabla) = B(\oldz, \nabla) $ implies $k = 0$,  the core acts freely on $\partial \Omega$. The result now follows from \cite[Proposition 5.5]{CL2}. 
\end{proof}

We can now recover \cite[Theorem 3.4]{cun2} from the analysis of \cite{CL2}.

\begin{corollary}[Cuntz]\label{cuntzsimple}
The $C^*$-algebra $\qn$ is simple and purely infinite.
\end{corollary}
\begin{proof}
The boundary quotient is simple and purely infinite by \cite[Theorem 5.1]{CL2}, so
the result follows from  \thmref{qnisboundaryquotient}.
\end{proof}

\begin{corollary}
There is a faithful representation $\pi$ of $\qn$ on $\ell^2(\Z)$ such that $\pi(s)e_n=e_{n+1}$ and $\pi(v_p)e_n=e_{pn}$. 
\end{corollary}

\begin{proof}
We define isometries $S$ and $V_p$ on $\ell^2(\Z)$ by $Se_n=e_{n+1}$ and $V_pe_n=e_{pn}$, and check easily that they satisfy (Q1), (Q2), (Q5) and (Q6). Thus \proref{presentqn} gives a representation $\pi$ of $\qn$ such that $\pi(s)=S$ and $\pi(v_p)=V_p$. Since $S\not=0$, the representation is certainly not $0$, and hence by Corollary~\ref{cuntzsimple} is faithful.
\end{proof}

\section{The phase transition theorem }\label{secKMS}

Standard arguments using the presentation in \thmref{toeplitzpresentation} show that there is a strongly continuous action  $\sigma$ of $\R$ on $\TT(\nxnx)$ such that 
\begin{equation}\label{defsigma}
\sigma_t(s) = s\ \text{ and }\ \sigma_t(v_p) = p^{it} v_p\text{ for $p\in \primes$ and $t\in \R$.}
\end{equation} 
The action $\sigma$ is spatially implemented in the identity representation of
 $\TT(\nxnx)$ on $\ell^2(\nxnx)$  by the unitary representation $U:\R\to {\mathcal U}(\ell^2(\nxnx))$ defined in terms of the usual basis by
\[
U_t e_{ (m,a)} : = a^{it} e_{(m,a)}.
\]
Our goal in this section is to describe the equilibrium states of the system  $(\TT(\nxnx), \R,\sigma)$, which we do in \thmref{maintheorem},
and to discuss the implications of this theorem for the interplay between equilibrium and symmetries. The notion of equilibrium appropriate in
this context is that of KMS states; since there are some subtleties involved, we begin by recalling the relevant definitions.

Suppose that $\alpha$ is an action of $\R$ on a $C^*$-algebra $A$. An element $a$ of $A$ is \emph{analytic} for the action $\alpha$ if the function $t\mapsto\alpha_t(a)$ is the restriction to $\R$ of an entire function on $\C$; it is shown at the start of \cite[\S8.12]{ped}, for example, that the set $A^{\textnormal{a}}$ of analytic elements is a dense $*$-subalgebra of $A$. For $\beta\in (0,\infty)$, a state $\phi$ of $A$  is a \emph{KMS state at inverse temperature $\beta$ for $\alpha$}, or a \emph{KMS${}_\beta$ state for $\alpha$}, if it satisfies the following \emph{KMS$_\beta$ condition}:
\begin{equation}\label{defKMS}
\phi(dc) = \phi(c\alpha_{i\beta}(d))\ \text{ for all $c,d\in A^{\textnormal{a}}$.}
\end{equation}
In fact, it suffices to check \eqref{defKMS} for a set of analytic elements which spans a dense subspace of $A$ \cite[Proposition~8.12.3]{ped}, and hence this definition agrees with the one used in \cite[\S5.3]{bra-rob}. For $\beta>0$, every KMS${}_\beta$ state is $\alpha$-invariant \cite[Proposition~8.12.4]{ped}; for a state $\phi$ to be a KMS${}_0$ state, it is standard to require that $\phi$ satisfies \eqref{defKMS} with $\beta=0$ (so that $\phi$ is a trace), and that $\phi$ is $\alpha$-invariant.

For every system $(A,\R,\alpha)$, the set $K_\beta$ of KMS$_\beta$ states is a compact convex subset of the state space $S(A)$. The affine structure of the set $K_\beta$ is studied in \cite[\S5.2.3]{bra-rob}: it is always a simplex in the sense of Choquet, and the extremal KMS$_\beta$ states (that is, the extreme points of $K_\beta$) are always factor states. The same section in \cite{bra-rob} also discusses the relationship between KMS states and equilibrium states in models from quantum statistical mechanics.

For $\beta = \infty$ there are two possible notions of equilibrium. Following
Connes and Marcolli \cite[Definition 3.7]{CM2}, we distinguish between 
 \emph{KMS$_\infty$ states},
which are by definition the weak* limits of nets $\phi_i$ of KMS$_{\beta_i}$ states with $\beta_i\to\infty$, and \emph{ground states}, 
which are by definition states $\phi$ for which the  entire functions 
\[
z \mapsto \phi(d\alpha_z(c))\ \text{ for $c,d\in A^{\textnormal{a}}$}
\]
are bounded on the upper half-plane. With this distinction in mind,  \cite[Proposition~5.3.23]{bra-rob} and \cite[Proposition 3.8]{CM2}
 imply that the KMS$_\infty$ states form a compact convex subset of the ground states. As observed by Connes and Marcolli \cite[page 447]{CM2}, ground states need not be KMS$_\infty$
states, and our system provides another example of this phenomenon, see parts (3) and (4) of
Theorem 7.1 below. We point out that this relatively recent distinction
was not made in  \cite{bra-rob} or \cite{ped}, where the terms ``ground state'' and ``KMS${}_\infty$ state'' are used interchangeably to refer to the ground states of \cite[Definition 3.7]{CM2}.
 The definition of ground state in \cite{ped} looks different: there it is required that the functions $z \mapsto \phi(d\alpha_z(c))$ are bounded by $\|c\|\,\|d\|$. However, as pointed out in the proof of \cite[Proposition~5.3.19, $(2)\Longrightarrow(5)$]{bra-rob}, a variant\footnote{One suitable variant is formulated as Exercise~9 on page 264 of \cite{rud}.} of the Phragmen-Lindel\"of theorem implies that an entire function which is bounded on the upper half-plane is bounded by the sup-norm of its restriction to the real axis, which in this case is at most $\|c\|\,\|d\|$. It follows from the definition in \cite{ped} that it suffices to check boundedness for a set of elements which spans a dense subspace of $A^{\textnormal{a}}$.

For our system $(\TT(\nxnx),\R,\sigma)$, the spanning elements $s^mv_av_b^*v^{*n}$ for $\TT(\nxnx)$ satisfy 
\[
\sigma_t(s^mv_av_b^*s^{*n})=(ab^{-1})^{it}s^mv_av_b^*s^{*n},
\]
and hence are all analytic. Thus to see that a state $\phi$ of $\TT(\nxnx)$  is a KMS${}_\beta$ state or ground state for $\sigma$, it suffices to check the appropriate condition for $c$ and $d$ of the form $s^mv_av_b^*s^{*n}$.   

We can now state our main theorem. The function $\zeta$ appearing in the formulas is the Riemann zeta-function, defined for $r>1$ by $\zeta(r)=\sum_{n=1}^\infty n^{-r}$.

\begin{theorem}\label{maintheorem}
Let $\sigma$ be the dynamics on $\TT(\nxnx)$ which satisfies \eqref{defsigma}. 
\begin{enumerate}
\item For $\beta\in [0,1)$ there are no KMS$_\beta$ states for $\sigma$.

\smallskip
\item For $\beta \in [1,2]$ there is a unique KMS$_\beta$ state $\psi_\beta$ for $\sigma$, and it is characterised by \[\psi_\beta (s^mv_av_b^*s^{*n}) = \begin{cases} 0 & \text{ if } a \neq b\text{ or } m \neq n\\
a^{-\beta} &\text{ if } a = b \text{ and } m=n.
\end{cases}
\]  

\smallskip
\item For $\beta \in (2,\infty]$, the simplex of KMS$_\beta$ states for $\sigma$ is isomorphic to the simplex of probability measures on $\T$; for $z\in\T$, the extremal KMS$_\beta$ state $\psi_{\beta,\newz}$ corresponding to the point mass $\delta_z$ is a type I factor state satisfying
\[
 \psi_{\beta, \newz}(s^m v_a v_b^* s^{*n}) = \begin{cases}  \displaystyle 0 & \text{ if } a \neq b  \text{ or } m \not\equiv n \pmod a,\\ \displaystyle
 \frac{1}{a \zeta(\beta -1)} \sum_{\{x\,:\,a \mid x \mid (m-n)\}} x^{1-\beta} z^{({m-n})/{x}} & \text { if }  a = b \text{ and } m \equiv n    \pmod  a.
\end{cases}
\]

\item If $\psi$ is a ground state for $\sigma$, then the restriction $\omega:=\psi|_{C^*(s)}$ is a state of $C^*(s)\cong \TT(\N)$, and we have
\begin{equation}\label{formground}
\psi(s^m v_a v_b^* s^{*n}) = 
\begin{cases} 0 & \text{ unless } a = b = 1\\
\omega(s^ms^{*n}) & \text{ when } a = b = 1.
\end{cases}
\end{equation}
The map $\psi \mapsto \psi|_{C^*(s)}$ is an affine isomorphism of the compact convex set of ground states onto the state space of $\TT(\N)$, and a state $\psi$ is an extremal ground state if and only if $\psi|_{C^*(s)}$ is either a vector state of $\TT(\N)$ or is lifted from an evaluation map on the quotient $C(\T)$ of $\TT(\N)$.
\end{enumerate}
\end{theorem}

We will prove these assertions in the next three sections. Before we start, though, we make some 
remarks on the significance of the theorem for symmetries and equilibrium.

\begin{remarks} (i) Although
a KMS$_\beta$ state $\psi$ (and in fact any state of $\TT(\nxnx)$) is uniquely determined by its value on products of the form
$s^{m} v_av_b^* s^{*n}$,  it is not obvious that there are states satisfying the formulas in parts (2), (3), and (4). We will prove existence of such states in \S\ref{secconstructionKMSground} using spatial constructions.

\smallskip
(ii) The $C^*$-algebra $\TT(\nxnx)$ carries a dual action $\hat\tau$ of $({\qx})^\wedge$, which is characterised on generators by $\hat\tau_\gamma(s)=s$ and $\hat\tau_\gamma(v_p)=\gamma(p)v_p$, and a dual coaction $\delta$ of $\qxqx$ (see \cite[Proposition~6.1]{quasilat}). It may help coaction fans to observe that $\hat\tau$ is the action of $({\qx})^\wedge$ corresponding to the restriction $\delta|$ of the coaction $\delta$ to the quotient $\qx=(\qxqx)/\Q$. The coaction $\delta$ gives an expectation $E_\qxqx$ onto the fixed-point algebra $\clsp\{s^{m} v_av_a^* s^{*m}:m \in \N, a \in \nx\}$, which is faithful because $\qxqx$ is amenable (see \cite[Lemma~6.5]{quasilat}), and $\hat\tau$ gives a faithful expectation $E_{\hat\tau}$ onto 
\[
\TT(\nxnx)^{\hat\tau}=\clsp\{s^{m} v_av_a^* s^{*n}: m, n \in \N, a \in \nx\}.
\]
 The dynamics $\sigma$ is the composition of $\hat\tau$ with the embedding $t\mapsto (\gamma_t:r\mapsto r^{it})$ of $\R$ as a dense subgroup of $({\qx})^\wedge$. So $\TT(\nxnx)^{\hat\tau}=\TT(\nxnx)^{\sigma}$, and $E_{\hat\tau}$ is also the expectation onto the fixed-point algebra for $\sigma$.

\smallskip
(iii) It follows from \thmref{maintheorem}
that  KMS$_\beta$ states vanish on the products $s^mv_av_b^*s^{*n}$ with $a \neq  b$, and hence factor through
the conditional expectation $E_{\hat\tau}$ of the dual action of $({\qx})^\wedge$; for $\beta\in [1,2]$ they  also vanish on the products 
$s^{m} v_av_a^* s^{*n}$ with $m \neq n$,  and hence
factor  through the conditional expectation $E_\qxqx$ of the dual coaction of $\qxqx$. 
Hence, for small $\beta$, the equilibrium states are symmetric with respect to the coaction of $\qxqx$
but for $\beta > 2$ they are symmetric only with respect to the (quotient) coaction of $\qx$. 
Since the extreme states in part (3) are indexed by the circle,
there is a circular symmetry at the level of KMS states which is broken  as $\beta$ increases through~$2$.

\smallskip
(iv) The relation (T1) makes it unlikely for there to be an action of the 
Pontryagin dual $\T$ of $\Z$ on $\TT(\nxnx)$ that sends $s\mapsto zs$ for $z\in \T$, and certainly not one which has any $v_p$ as an eigenvector. Thus the symmetry which is apparently being broken as $\beta$ passes from $2^-$ to $2^+$ in \thmref{maintheorem} does not obviously come from a group action on the $C^*$-algebra $\TT(\nxnx)$.

\smallskip
(v) There is a further phase transition ``at infinity": the KMS$_\infty$ states form a proper subset of the ground states. Indeed, it follows from the formula in (3) that every KMS${}_\infty$ state satisfies $\psi(ss^*)=1$, and hence the extremal KMS$_\infty$ states are the ground states such that $\psi|_{C^*(s)}$ is lifted from an evaluation map on $C(\T)$. Notice also that the existence of the affine isomorphism in (4) implies that the ground states of $\TT(\nxnx)$ do not form a simplex, because the state space of the noncommutative subalgebra $C^*(s)\cong\TT(\N)$ is not a simplex (see, for example, \cite[Example 4.2.6]{bra-rob}).

\smallskip
(vi) (The partition function.) The extremal KMS$_\beta$ states (for $\beta>2$) are related to the KMS$_{\infty}$ states in the following way. Since each extremal KMS$_{\infty}$ state $\phi$ is $\R$-invariant, the dynamics is implemented in the GNS-representation $(\HH_\phi,\pi_\phi,\xi_\phi)$ by a unitary group $U:\R\to U(\HH_\phi)$. The \emph{Liouville operator} is the infinitesimal generator $H$ of this one-parameter group, which is an unbounded self-adjoint operator on $\HH_\phi$. The functionals 
\[
\phi_\beta:T\mapsto \frac{\tr (e^{-\beta H}T)}{\tr e^{-\beta H}}
\]
are then the extremal KMS$_\beta$ states; the normalising factor $\beta\mapsto \tr e^{-\beta H}$ is called the \emph{partition function} of the system. On the face of it, the partition function will depend on the choice of KMS$_{\infty}$ state $\phi$, but in these number-theoretic systems it doesn't seem to. In the Bost-Connes system, for example, there is a large symmetry group of the underlying $C^*$-algebra which commutes with the dynamics and acts transitively on the extreme KMS$_\infty$ states, and all the Liouville operators in the associated GNS representations match up (see \cite[\S6]{bos-con}). The same thing happens for similar systems over more general number fields (see \cite[Remark~3.5]{LvF}). Here, even though there is no obvious symmetry group of $\TT(\nxnx)$ which implements the circular symmetry of the simplex of KMS$_\infty$ states, the GNS representations of the extreme KMS$_\infty$ states $\psi_{\infty,z}$ are all realisable on the same space $\ell^2(X)$, with the same cyclic vector $e_{0,1}$, the same unitary group implementing the dynamics, and the same Liouville operator (see the discussion at the start of the proof of Proposition~\ref{constructKMS>2}). That discussion shows also that the eigenvalues of $H$ are the numbers $\ln x$ for $x\in \nx$, and that the multiplicity of the eigenvalue $\ln x$ is $x$, so that $\tr e^{-\beta H}=\zeta(\beta-1)$. So it makes sense for us to claim that: ``The partition function of the system $(\TT(\nxnx),\R,\sigma)$ is $\zeta(\beta-1)$.''
\end{remarks}

\section{Characterisation of KMS and ground states of the system}\label{seccharacterisationKMSground}
 We begin with the case $\beta <1$. 

\begin{proposition}[\thmref{maintheorem}(1)]\label{noequilibriumatlowbeta}
The system $(\TT(\nxnx),\R, \sigma)$ has
no KMS$_\beta$ states for $\beta<1$.
\end{proposition}
\begin{proof}
(Notice that our argument also rules out the existence of KMS$_\beta$ states for $\beta<0$.)
Suppose $\psi$ is a KMS$_\beta$ state for $\sigma$. Then the KMS$_\beta$ condition implies that, for $a\in\nx$ and $0\leq k<a$, we have
\[
\psi (s^{k} v_av_a^* s^{*k}) =\psi (v_a^* s^{*k}\sigma_{i\beta}(s^{k} v_a))=\psi (v_a^* s^{*k}a^{-\beta}s^{k} v_a)=a^{-\beta}\psi(1)=a^{-\beta}.
\]
The relation (T5) (or strictly speaking, (T5') in Lemma~\ref{relsata}) implies that the projections $ s^{k} v_av_a^* s^{*k}$ for  $0 \leq k <a$ are mutually orthogonal, and hence $1 \geq \sum_{k=0} ^{a-1} s^{k} v_av_a^* s^{*k}$. Now positivity of $\psi$ implies that
\[
1=\psi(1)\geq \psi \Big(\sum_{k=0} ^{a-1} s^{k} v_av_a^* s^{*k}\Big)=aa^{-\beta},
\]
which implies $\beta\geq 1$.
\end{proof}

For $\beta \geq 1$ we need the characterisation of the KMS$_\beta$ states in Lemma~\ref{KMScharacterisationlemma}. Here and later we use the following notational convention to simplify formulas. 
\begin{convention}\label{convention(())}
 We write $s^{((k))}$
to mean $s^k$ when $k\geq 0$ and $s^{*(-k)}$ when $k<0$.
\end{convention}

\begin{lemma}\label{KMScharacterisationlemma}
Let $\beta \in [1,\infty)$. A state $\phi$  of $\TT(\nxnx)$ is
 a KMS$_\beta$ state for $\sigma$ if and only if for  every $a, b \in \nx$ and $m,n \in \N$ we have 
\begin{equation}\label{KMScharacterisation}
\phi(s^m v_a v_b^* s^{*n}) = \begin{cases}  0 & \text{ if } a\neq b  \text{ or } m \not\equiv n \pmod a\\
 a^{-\beta} \phi\big(s^{((\frac{m-n}{a}))}\big) & \text { if }  a = b \text{ and } n \equiv m    \pmod  a . \end{cases}
\end{equation}
\end{lemma}

\begin{proof}
Suppose first that $\phi$ is a KMS$_\beta$ state. 
Applying the KMS condition twice gives 
\[
\phi(s^m v_a v_b^* s^{*n}) = a^{-\beta} \phi( v_b^* s^{*n} s^m v_a) = (a/b)^{-\beta} \phi(s^m v_a v_b^* s^{*n}),
\]
which implies that 
\begin{equation}\label{intermedcalc}
\phi (s^mv_av_b^*s^{*n}) = 
\begin{cases}0&\text{ if $a\neq b$}\\
a^{-\beta} \phi (v_a^* s^{((m-n))} v_a )&\text{ if $a= b$.}
\end{cases} 
\end{equation}
When  $m\not\equiv n \pmod a$, the relation (T5) implies that $v_a^* s^{((m-n))} v_a=0$, 
and  when  $m \equiv n \pmod a$, relation (T1) implies that $v_a^* s^{((m-n))} v_a =s^{(((m-n)/a))}$. Thus \eqref{intermedcalc} says that $\phi$ satisfies \eqref{KMScharacterisation}.

Suppose now that $\phi$ satisfies \eqref{KMScharacterisation}. Since it suffices to check the KMS condition 
\eqref{defKMS} on spanning elements, $\phi$ is a KMS${}_\beta$ state for $\sigma$ and if and only if 
\begin{equation}\label{KMSanalytic}
{a}^{\beta}\phi(s^m v_av_b^* s^{*n} \ s^q v_c v_d^* s^{*r} ) =  
{b}^{\beta} \phi( s^q v_c v_d^* s^{*r}\ s^m v_av_b^* s^{*n} )
\end{equation}
for $a, b, c, d \in \nx$ and $m,n,q,r \in \N$. We prove this equality by computing both sides.

To compute the left-hand side of \eqref{KMSanalytic}, we first reduce 
the expression using the covariance relation in \lemref{covarianceonngenerators}:
\begin{align*}
s^m v_av_b^* s^{*n} \ s^q v_c v_d^* s^{*r}  &=
	 s^m v_av_b^* s^{((q-n))} v_cv_d^*s^{*r} \\
	 &=   \begin{cases}
0 & \text{ if } q\not\equiv n \pmod {\gcd(b,c)} \\
s^m v_a    ( s^\beta    v_{c'} v^*_{b'}    s^{*\gamma} )   v_d^* s^{*r}   & \text{ if } q \equiv n \pmod {\gcd(b,c)} 
\end{cases} 
\\
	 &=   \begin{cases}
0 & \text{ if } q\not\equiv n \pmod {\gcd(b,c)} \\
s^{m+\beta a} v_{ac'} v^*_{db'}  s^{*r+\gamma d}  & \text{ if } q \equiv n \pmod {\gcd(b,c)} ,
\end{cases} 
	\end{align*} 
where $b' = b/\gcd(b,c)$,  $c' = c/\gcd(b,c)$, and $( \beta, \gamma)$ is the smallest non-negative solution of $(q-n)/{\gcd(b,c)} = \beta b' - \gamma c'$.  
Now \eqref{KMScharacterisation} implies that the left-hand side of \eqref{KMSanalytic} vanishes unless 
$ q \equiv n \pmod{\gcd(b,c)}$, $ac' = db'$,  and $m+\beta a\equiv r + \gamma d \pmod {ac'}$, 
in which case it equals
\begin{equation}\label{lhsofkms}
(c')^{-\beta}\phi\big(  s^{((\frac{m+\beta a-r-\gamma d}{ac'}))}\big) .
\end{equation}

The  analogous computation shows that the right-hand side of
 \eqref{KMSanalytic} vanishes unless
$ m \equiv r \pmod {\gcd(d,a)}$ and $ca' = bd'$. If so, we take $(\delta, \alpha)$ to be the smallest non-negative solution of $(m-r)/\gcd(d,a) = \delta d' - \alpha a'$. Now the right-hand side of \eqref{KMSanalytic} vanishes unless $q + \delta c\equiv n + \alpha b\pmod {bd'}$, and then equals
\begin{equation}\label{rhsofkms}
(d')^{-\beta}\phi(  s^{((\frac{q + \delta c- n - \alpha b}{bd'}))} ).
\end{equation}

We need to verify that the conditions for a nonvanishing  left-hand side 
match those for the right-hand side,
and that when they hold, the values of \eqref{lhsofkms} and \eqref{rhsofkms} coincide. The situation is symmetric, so we suppose that $q \equiv n \pmod {\gcd(b,c)}$, that $ac' = db'$, and that, with $(\beta, \gamma)$ as defined two paragraphs above, $m+\beta a\equiv r + \gamma d \pmod {ac'}$.

Notice that 
\[
ac' = db' \iff a/d = b'/c' \iff a'/d' = b'/c' \iff a'/d' = b/c \iff ca' = bd'; 
\]
these are all equivalent to $ac = bd$, and from the reduced form in the middle we deduce that $a' = b'$
and $c' = d'$. This implies in particular that the coefficients $(c')^{-\beta}$ in \eqref{lhsofkms} and $(d')^{-\beta}$ in \eqref{rhsofkms} coincide.
Next, notice that $m+\beta a\equiv r + \gamma d \pmod {ac'}$ implies that $m\equiv r \pmod {\gcd(d,a)}$, so it makes sense to take $(\delta, \alpha)$ to be the smallest non-negative solution of $(m-r)/{\gcd(d,a)} = \delta d' - \alpha a'$.

Consider now the exponent of $s$ on the left-hand side of \eqref{KMSanalytic}. The definition of $(\delta,\alpha)$ implies that $m-r=\delta d - \alpha a$, so, remembering that $a' = b'$ and $c' = d'$, we have
\begin{align*}
\frac{m+\beta a-r-\gamma d}{ac'} 
&= \frac{ (\delta - \gamma) d + (\beta - \alpha) a}{\gcd(d,a)a'c'}
=\frac{ (\delta - \gamma) d' + (\beta - \alpha) a'}{a'c'}\\
&=\frac{ (\delta - \gamma) c' + (\beta - \alpha) b'}{b'd'}
=\frac{ \beta b - \gamma c + \delta c- \alpha b}{\gcd(b,c)b'd'}\\
&= \frac{q-n+\delta c- b\alpha}{bd'},
\end{align*}
which is the exponent of $s$ on the right-hand side of \eqref{KMSanalytic}.
Since $ac'$ divides $m+\beta a -r-\gamma d$, this calculation also shows that
$b'd$ divides $q + \delta c- n - \alpha b$, or equivalently that $q + \delta c\equiv n + \alpha b\pmod {bd'}$. This completes the proof of \eqref{KMSanalytic}, and hence we have shown that $\phi$ is a KMS$_\beta$ state.
\end{proof}

\begin{lemma}\label{lemmagroundcharacterisation}
A state $\phi$ of $\TT(\nxnx)$ is a ground state for $\sigma$ if and only if 
\begin{equation}\label{groundcharacterisation}
\phi( s^m v_a v_b^* s^{*n}) = 0 \ \text{ whenever $a\not=1$ or $b\not=1$. }
\end{equation}
\end{lemma}

\begin{proof} 
Let $\phi$ be a state of $\mathcal T (\nxnx)$. 
 The expression 
\begin{equation*}
\phi( s^q v_c v_d^* s^{*r} \ \sigma _{\alpha + i\beta}( s^m v_a v_b^* s^{*n})) =  (a/b)^{i\alpha - \beta} \phi( s^q v_c v_d^* s^{*r} \ s^m v_a v_b^* s^{*n}) 
\end{equation*}
is bounded on the upper half plane ($\beta >0$)  if and only if 
\begin{equation}\label{groundcharact-old}
\phi( s^q v_c v_d^* s^{*r} \ s^m v_a v_b^* s^{*n}) = 0 \quad \text{ whenever } a < b .
\end{equation}
Suppose $\phi$ is a ground state and  choose
 $ r = m$ and  $d = a = 1$;
 then \eqref{groundcharact-old} implies $\phi( s^q v_c v_b^* s^{*n}) = 0 $
 for $1 < b$. Taking adjoints gives the same for $1 < c$. This proves \eqref{groundcharacterisation}
 (with $q$ in place of $m$ and $c$ in place of $a$).
  Conversely, suppose  $\phi( s^m v_a v_b^* s^{*n}) = 0 $ whenever $a$ or $b$ is not $1$ and 
choose two analytic elements $X= s^m v_av_b^* s^{*n}$ and $Y$ for $\sigma$; then 
the Cauchy-Schwarz inequality yields
\begin{align*}
| \phi (Y^* \sigma _{\alpha + i\beta}(s^m v_av_b^* s^{*n} ) )|^2
& =
|(a/b)^{i\alpha -\beta} \phi(Y^* s^m v_av_b^* s^{*n} ) |^2 \\
& \leq
(a/b)^{-\beta} \phi(Y^* Y) \phi(s^n v_bv_a^* s^{*m} \  s^m v_av_b^* s^{*n} )\\
& = (b/a)^{\beta} \phi(Y^* Y) \phi(s^n v_b v_b^* s^{*n} ).
\end{align*}
Since the last factor vanishes for $b \neq 1$, the function 
$\alpha + i \beta \mapsto \phi(Y^* \sigma _{\alpha + i\beta}(X))$ is bounded for $\beta >0$, 
so $\phi $ is a ground state.
\end{proof}

\section{Construction of KMS and ground states}\label{secconstructionKMSground}

To prove that there exists a KMS$_\beta$ state satisfying the formula in part (2) of \thmref{maintheorem}
we use a product measure arising on the factorization $\Omega_B\cong \prod_{p\in \primes} X_p$ of \proref{prodstructure}.
The construction makes sense for $\beta \geq 1$, but the case $\beta = 1$ requires special consideration.

\begin{proposition}[\thmref{maintheorem}(2): existence of a KMS$_\beta$ state that factors through $E_\qxqx$]
\label{productmeasure} For $k\in \N$ and $r\in \Z/p^k$, let $\delta_{(\oldz,p^k)}$ denote the unit point mass at $B(\oldz,p^k) \in X_p$.
For $\beta >1$, the series 
\[
\mu_{\beta,p}  = (1-p^{1-\beta}) \sum_{(\oldz,p^k)\in X_P} p^{- \beta k} \delta_{(\oldz,p^k)}
\]
defines a Borel probability measure on $X_p$; for $\beta = 1$, we let $\mu_{1,p}$ be the probability measure on $X_p$ coming from
additive Haar measure on $\Z_p$ via the embedding $r\mapsto B(r,p^\infty)$ of $\Z_p$ in $X_p$ (see Lemma~\ref{omegaB}). Let $\mu_\beta$ be the measure on $\Omega_B$ coming from the product measure $\prod_{p\in\primes}\mu_{\beta,p}$ on $\prod_{p\in\primes} X_p$ via the homeomorphism of Proposition~\ref{prodstructure}, let $\mu_{\beta}^*:f\mapsto\int f\,d\mu_{\beta}$ be the associated state on $C(\Omega)$, and view $\mu_\beta^*$ as a state on 
$\TT(\nxnx)^\delta$ using the isomorphism of 
\[
\TT(\nxnx)^\delta=\clsp\{s^mv_av_a^*s^{*m}:(m,a)\in\nxnx\}
\]
onto $C(\Omega)$ which takes $s^mv_av_a^*s^{*m}$ to $1_{m,a}$. Then $\psi_\beta :=\mu_\beta^* \circ E_\qxqx$ is a KMS${}_\beta$ state for $1 \leq \beta \leq \infty$, and it satisfies
\begin{equation}\label{charpsibeta}
\psi_\beta(s^mv_av_b^*s^{*n})=\begin{cases}
0&\text{ unless $a=b$ and $m=n$}\\
a^{-\beta}&\text{ if $a=b$ and $m=n$.}
\end{cases}
\end{equation}
\end{proposition}

\begin{proof}
Suppose first that $1< \beta < \infty$. Then the series
\[
\sum_{(\oldz,p^k)\in X_P} p^{- \beta k} = \sum_{k \in \N} p^k p^{- \beta k}
\]
converges with sum $(1-p^{1-\beta})^{-1}$, so the sum defining $\mu_{\beta,p}$ converges in norm in $M(\Omega)$ to a probability measure. 

To prove that $\psi_\beta$ is a KMS${}_\beta$ state, we compute $\psi_\beta (s^m v_a v_b^* s^{*n}) $ and apply Lemma~\ref{KMScharacterisationlemma}. Since $\psi_\beta$ factors through $E_\qxqx$, we have
$\psi_\beta (s^m v_a v_b^* s^{*n}) =0 $ whenever $m\neq n$ or $a\neq b$.
So suppose that $m=n$ and $a = b= \prod_{p|a} p^{e_p(a)}$. The isomorphism of $\TT(\nxnx)^\delta$ with $C(\Omega)$ carries $s^mv_av_a^*s^{*m}$ into $1_{m,a}$, which is the characteristic function of the set
$\{B(r,N): a|N,\ r(a)=m(a)\}$; the homeomorphism of Proposition~\ref{prodstructure} carries this set into
\begin{equation}\label{supp1ma}
\Big(\prod_{p|a}\{B(r,p^k):k\geq e_p(a),\ r(p^{e_p(a)})=m(p^{e_p(a)}) \}\Big)\times\Big(\prod_{q\notdiv a}X_q\Big).
\end{equation}
Thus 
\begin{align*}
\psi_\beta (s^m v_a v_a^* s^{*m})
&=\int 1_{m,a}\,d\mu_\beta=\mu_\beta\big(\{B(r,N):a|N,\ r(a)=m(a)\}\big)\\
&=\prod_{p|a}\mu_{\beta,p}\big(\{B(r,p^k):k\geq e_p(a),\ r(p^{e_p(a)})=m(p^{e_p(a)})\} \big)\times\Big(\prod_{q\notdiv a}\mu_{q,\beta}(X_q)\Big)\\
&=\prod_{p|a}\mu_{\beta,p}\big(\{B(r,p^k):k\geq e_p(a),\ r(p^{e_p(a)})=m(p^{e_p(a)}) \}\big)\\
&=\prod_{p|a}(1-p^{1-\beta})\Big(\sum_{k=e_p(a)}^\infty p^{-\beta k}\big(\#\{r\in \Z/p^k:r(p^{e_p(a)})=
m(p^{e_p(a)})\}\big)\Big).
\end{align*}
For $k\geq e_p(a)$ there are $p^{k-e_p(a)}$ elements $r$ in $\Z/p^k$ such that $r(p^{e_p(a)})=m(p^{e_p(a)})$. Thus
\begin{align*}
\psi_\beta (s^m v_a v_a^* s^{*m})
&=\prod_{p|a}(1-p^{1-\beta})\Big(\sum_{k=e_p(a)}^\infty p^{(1-\beta)k}p^{-e_p(a)}\Big)\\
&=\prod_{p|a}(1-p^{1-\beta})p^{-\beta e_p(a)}\Big(\sum_{l=0}^\infty p^{(1-\beta)l}\Big)\\
&=\prod_{p|a}p^{-\beta e_p(a)}=\Big(\prod_{p|a}p^{e_p(a)}\Big)^{-\beta }=a^{-\beta}.
\end{align*}
Since the expectation $E_\qxqx$ kills the nonzero powers of $s$, this calculation shows that $\psi_\beta$ satisfies \eqref{KMScharacterisation}, and hence Lemma~\ref{KMScharacterisationlemma} implies that $\psi_\beta$ is a KMS${}_\beta$ state.

Now suppose $\beta = 1$. Then the measure $\mu_{1}$ is the product of normalized Haar measures on the $\Z_p$, which is the normalised Haar measure on $\hatz \cong \partial\Omega$. This satisfies $\mu_1(aE)= a\inv\mu_1(E)$, and since the support of $1_{m,a}$ is $m+a \hatz$, we have
\[
\psi_1(s^m v_a v_a^* s^{*m})=\int_{\hatz} 1_{m,a}\,d\mu_1=\mu_1(m+a\hatz)=a\inv\mu_1(\hatz)=a\inv.
\]
So Lemma~\ref{KMScharacterisationlemma} also implies that $\psi_1$ is a KMS${}_1$ state. 

When $\beta = \infty$, the usual interpretation $a^{-\infty} = 0$ for $a>1$ and $1^{-\infty} = 1$ yields probability  measures $\mu_{\infty, p}$ on $X_p$ 
 concentrated at the point $(0,1)\in X_p$, and their product corresponds to the unit point mass $\mu_\infty$ concentrated at the point $B(0,1) \in \Omega_B$.
Then $\psi_\infty := (\mu_\infty )_* \circ E_\qxqx$ satisfies
 \begin{equation}\label{charpsiinfty}
\psi_\infty (s^mv_av_b^*s^{*n})=\begin{cases}
0&\text{ unless $a=b =1$ and $m=n$,}\\
1&\text{ if $a=b =1$ and $m=n$.}
\end{cases}
\end{equation}
and is a ground state by \lemref{groundcharacterisation}.
The characterisations of $\psi_\beta$ and $\psi_\infty$ show that 
 $\psi_\beta (c) \to \psi_\infty (c)$ as $\beta \to \infty$ for $c=s^mv_av_b^*s^{*n}$, and hence $\psi_\infty$ is a KMS$_\infty$ state.  
\end{proof}

To construct KMS$_{\beta}$ states for $\beta>2$, we use the Hilbert-space representation of $\nxnx$ described in the next lemma. For $2<\beta<\infty$, the state $\omega$ in the lemma will be lifted from a state on the quotient $C(\T)$ of $\TT(\N)$, hence given by a probability measure $\mu$ on $\T$, and then the isometry $U$ in the $GNS$ representation is the multiplication operator $(Uf)(z)=zf(z)$ on $L^2(\T,d\mu)$. When we construct ground states, $\omega$ can be any state of $\TT(\N)$.

\begin{lemma}\label{constructreps}
Let $\omega$ be a state of the Toeplitz algebra $\TT(\N)$, and let $U$ be the generating isometry for the GNS representation $(\HH_\omega,\pi_\omega,\xi_\omega)$ of $\TT(\N)$. Set 
\[
X:=\{(r,x):x\in \nx,\ r\in \Z/x\}
\]
and let $e_{r,x}$ be the usual basis for $\ell^2(X)$.
Let $S$ and $V_p$ be the isometries on $\ell^2(X,\HH_\omega)$ which are characterised by the following behaviour on elements of the form $fe_{(r,x)}$ for $f\in \HH_\omega$:
\begin{align*}
S(fe_{\oldz,x})&=\begin{cases} f e_{\oldz +1,x} & \text{ if } \oldz + 1 \not=0_{\Z/x},\
\\(Uf)e_{0,x} & \text{ if } \oldz + 1 =0_{\Z/x}, \text{ and}\end{cases}\\
V_p(fe_{\oldz,x})&=fe_{p\oldz,px}.
\end{align*}
Then $S$ and $\{V_p:p\in\primes \}$ satisfy the relations \textnormal{(T1)--(T5)} of \thmref{toeplitzpresentation}.
\end{lemma}

\begin{proof}
To verify (T1), first observe that 
\begin{equation}\label{totp1}
V_p S(fe_{\oldz,x}) 
= \begin{cases}  fe_{p\oldz +p,px} & \text{ if } \oldz + 1\neq 0_{\Z/x}\\
(Uf)e_{0,px} & \text{ if } \oldz + 1 = 0_{\Z/x}.\end{cases}
\end{equation}
To compute $S^pV_p$, first note that for $k\leq x$ we have
\[
S^k(fe_{\oldz,x})
=\begin{cases}
fe_{\oldz +k,x} & \text{ if $\oldz +i\neq 0_{\Z/x}$ for $i$ satisfying $0<i\leq k$,} \\
(Uf)e_{k-i,x} & \text{ if there exists $i$ such that $0<i\leq k$ and $\oldz +i= 0_{\Z/x}$,}\end{cases}
\]
which, since $p\leq px$, implies that
\begin{align*}
S^p  V_p (fe_{\oldz,x})
&=\begin{cases}  fe_{p\oldz +p,px} & \text{ if $p\oldz +i\neq 0_{\Z/px}$ for $i$ satisfying $0<i\leq p$,} \\
(Uf)e_{p-i,px} & \text{ if there exists $i$ such that $0<i\leq p$ and $p\oldz +i= 0_{\Z/px}$}\end{cases}\\
&=\begin{cases} fe_{p\oldz +p,px} & \text{ if $p\oldz +p\neq 0_{\Z/px}$,} \\
(Uf)e_{0,px} & \text{ if $p\oldz +p= 0_{\Z/px}$,}\end{cases}
\end{align*}
which is the same as \eqref{totp1} because $\times p:\Z/x\to \Z/px$ is injective (see \lemref{timesb}). Thus (T1) holds.

To verify relation (T2), we just need to observe that $(\times p)\circ (\times q)=(\times q)\circ (\times p)$: indeed, both are just $\times pq:\Z/x\to \Z_{pqx}$. For (T3), we suppose that $p$ and $q$ are distinct primes. The adjoint $V_p^*$ is given by
\[
V_p^*(fe_{\oldz,x})
=\begin{cases}  fe_{w,p^{-1}x} & \text{ if $p|x$ and $\oldz =pw$ for some $w\in \Z/p^{-1}x$,} \\
0 & \text{ otherwise.}\end{cases}
\]
Thus we have 
\begin{equation}\label{vp*vq}
V_p^*V_q(fe_{\oldz,x})=
\begin{cases}  fe_{w,p^{-1}qx} & \text{ if $p|qx$ and $q \oldz =pw$ for some $w\in \Z/p^{-1}qx$,} \\
0 & \text{ otherwise,}\end{cases}
\end{equation}
whereas
\begin{equation}\label{vqvp*}
V_qV_p^*(fe_{\oldz,x})=
\begin{cases}  fe_{q\zeta,qp^{-1}x} & \text{ if $p|x$ and $\oldz =p\zeta$ for some $\zeta\in \Z/p^{-1}x$,} \\
0 & \text{ otherwise.}\end{cases}
\end{equation}
We know from \lemref{timesb} that $\oldz=p\zeta\Longleftrightarrow \oldz\equiv 0\pmod p$, which is equivalent to $q \oldz\equiv 0\pmod{p}$ because $\gcd(q,p)=1$ ; thus the non-trivial cases in \eqref{vp*vq} and \eqref{vqvp*} coincide, with $w=q\zeta$, and we have $V_p^*V_q=V_qV_p^*$, which is (T3).

To verify (T4), we first compute the left-hand side: 
\begin{equation}\label{lhst4} 
S^* V_p f e_{\oldz,x} = S^* f e_{p\oldz,px}  
=\begin{cases} f  e_{p\oldz -1, px} & \text{ if $p\oldz \neq 0_{\Z/px}$} \\
(U^*f) e_{px -1,px} & \text{ if  $p\oldz = 0_{\Z/px}$.}\end{cases}\\
\end{equation}
For the right hand side, we have
\begin{align*}
S^{p-1}  V_p S^* f e_{\oldz,x}
&=\begin{cases}  S^{p-1}  V_p fe_{\oldz -1 ,x} & \text{ if $\oldz \neq 0_{\Z/x}$} \\
S^{p-1} V_p   (U^*f)e_{x-1,x} & \text{if $\oldz = 0_{\Z/x}$}\end{cases}\\
&=\begin{cases}  S^{p-1}  fe_{p\oldz -p ,px} & \text{ if $\oldz \neq 0_{\Z/x}$} \\
S^{p-1}   (U^*f)e_{px-p,px} & \text{ if  $\oldz = 0_{\Z/x}$}\end{cases}\\
&=\begin{cases} fe_{p\oldz -p +p-1,px} & \text{ if $p\oldz \neq 0_{\Z/px}$} \\
   (U^*f)e_{px-p+p-1,px} & \text{ if $p\oldz = 0_{\Z/px}$,}\end{cases}
\end{align*}
which is the same as \eqref{lhst4}.

Finally, we verify (T5). Suppose $1\leq k<p$. Then
\[
V_p^*S^kV_p(fe_{\oldz,x})
=\begin{cases}V_p^*(fe_{p\oldz +k,px}) & \text{ if $p\oldz +i\neq 0_{\Z/px}$ for $0<i\leq k$} \\
V_p^*((Uf)e_{k-i,x}) & \text{ if there exists $i$ such that $0<i\leq k$ and $p\oldz +i= 0_{\Z/px}$.}\end{cases}
\]
Since $0<k<p$, the second possibility does not arise. Thus
\[
V_p^*S^kV_p(fe_{\oldz,x})=\begin{cases}fe_{w,x}&\text{ if $p\oldz +k=pw$}\\
0&\text{ otherwise,}
\end{cases}
\]
which has to be $0$ because $p\oldz +k$ cannot be in the range of $\times p$ for $k$ in the given range. This confirms (T5), and completes the proof.
\end{proof}

We can now prove the existence of many KMS${}_\beta$ states for $\beta>2$. 

\begin{proposition}[\thmref{maintheorem}(3): KMS$_\beta$ states from probability measures  on $\T$]\label{constructKMS>2}
Suppose $\beta \in (2,\infty)$ and $\mu$ is a probability measure on $\T$.  Then there is a state $\psi_{\beta,\mu}$ of
$\TT(\nxnx)$ such that
\begin{equation}\label{KMScomputation}
\psi_{\beta,\mu}(s^m v_a v_b^* s^{*n}) = \begin{cases}  \displaystyle 0 & \text{ if } a \neq b  \text{ or } m \not\equiv n \pmod a\\ \displaystyle
 \frac{1}{a \zeta(\beta -1)} \sum_{a \mid x \mid (m-n)} x^{1-\beta} \int_{\T}\newz^{({m-n})/{x}}\,d\mu(\newz) & \text { if }  a = b \text{ and } m \equiv n    \pmod  a. 
\end{cases}
\end{equation}
There is also a KMS$_{\infty}$ state $\psi_{\infty,\mu}$ such that 
\begin{equation}\label{KMSinftycomputation}
\psi_{\infty,\mu}(s^m v_a v_b^* s^{*n}) = \begin{cases}  \displaystyle 0 & \text{ unless } a = b =1\\ 
\displaystyle \int_{\T} \newz^{m-n} \,d\mu(\newz) & \text { if }  a = b = 1.\end{cases}
\end{equation}
For $\beta \in (2,\infty]$, the correspondence $\mu \to \psi_{\beta,\mu}$ is an affine map of the set $P(\T)$ of probability measures on the unit circle into the simplex of KMS$_\beta$ states.
Moreover, the extremal states $\psi_{\infty,\newz}$ for $\newz \in \T$ are pure and pairwise inequivalent.
\end{proposition}

\begin{proof}
As anticipated before \lemref{constructreps}, we apply that Lemma to the state $\omega$ of $\TT(\N)$ lifted from the measure $\mu$ on $\T$, so the Hilbert space $\HH_\omega$ is $L^2(\T,d\mu)$ and $(Uf)(z)=zf(z)$. The resulting family $S,V$ gives us a representation $\pi_\mu:=\pi_{S,V}$ of $\TT(\nxnx)$ on the Hilbert space $\ell^2(X,L^2(\T,d\mu))$. We aim to use this representation to define the states $\psi_{\beta,\mu}$. For motivation, we suppose first that $\mu=\delta_z$; then $U$ is multiplication by $\newz$ on $\C$ and the Hilbert space is $\ell^2(X)$ with the usual orthonormal basis $\{e_{\oldz,x}:(\oldz,x)\in X\}$. In this special case, we can borrow a construction from \cite{bos-con}.

We first note that there is a unitary representation $W:\R\to U(\ell^2(X))$ such that $W_te_{r,x}=x^{it}e_{r,x}$, and this representation implements the dynamics $\sigma$ in the representation $\pi_\mu$ --- in other words, $(\pi_\mu,W)$ is a covariant representation of the system $(\TT(\nxnx),\R,\sigma)$. The infinitesimal generator $H$ of $W$ is the (unbounded) self-adjoint operator $H$ on $\ell^2(X)$ such that $W_t=e^{itH}$, and is diagonal with respect to the basis $\{e_{\oldz,x}\}$, with eigenvalues $\ln x $ of multiplicity $x$. Then $e^{-\beta H}$ is a positive bounded operator which is also diagonalised by the $e_{\oldz,x}$ and satisfies $e^{-\beta H}e_{\oldz,x}=x^{-\beta}e_{\oldz,x}$. Thus for $\beta>2$, $e^{-\beta H}$ is a trace-class operator with
\[
\tr e^{-\beta H} = \sum_{(\oldz,x)\in X} \langle x^{-\beta}e_{\oldz,x},e_{\oldz,x}\rangle =
\sum_{x\in \nx} x^{1 -\beta} = \zeta(\beta-1),
\]
and $\zeta(\beta- 1)^{-1} e^{-\beta H}$ is a bounded positive operator with trace one,
which defines a state $\psi_{\beta,\mu}$ on $\TT(\nxnx)$ through the representation $\pi_\mu$:
\begin{equation}\label{defpsit}
\psi_{\beta,\mu} (T):= \frac{1}{\zeta(\beta-1)} \tr(e^{-\beta H} \pi_\mu(T))=\frac{1}{\zeta(\beta-1)}\sum_{(\oldz,x)\in X}x^{-\beta}\langle\pi_\mu(T)e_{\oldz,x},e_{\oldz,x}\rangle.
\end{equation}

For more general $\mu$, we can still define $H$ formally by $H(fe_{\oldz,x})=(\ln x)fe_{\oldz,x}$, but now  $e^{-\beta H}$ is no longer trace-class. (If we view $H_\mu$ as $\ell^2(X)\otimes L^2(\T,d\mu)$, the new $e^{-\beta H}$ is the tensor product $e^{-\beta H}\otimes 1$ of the old one with the identity operator on $L^2(\T,d\mu)$, which is not trace-class unless $L^2(\T,d\mu)$ is finite-dimensional.) Nevertheless, using \eqref{defpsit} as motivation, we can still define $\psi=\psi_{\beta,\mu}$ using the elements $e_{r,x}=1e_{r,x}$ by 
\[
\psi(T)=\frac{1}{\zeta(\beta-1)}\sum_{(\oldz,x)\in X}x^{-\beta}\langle\pi_\mu(T)e_{\oldz,x},e_{\oldz,x}\rangle, 
\]
and verify directly that $\psi$ is a positive functional with $\psi(1)=1$, hence is a state. We want to show that $\psi$ is a KMS$_{\beta}$-state satisfying~\eqref{KMScomputation}.

We check \eqref{KMScomputation} first. We have
\[
\psi(s^m v_a v_b^* s^{*n})
=\frac{1}{\zeta(\beta-1)}\sum_{(\oldz,x)\in X}x^{-\beta}\langle V_b^* S^{*n}e_{\oldz,x},V_a^*S^{*m}e_{\oldz,x}\rangle.
\]
Now $V_b^*S^{*n}e_{\oldz,x}$ has the form $fe_{s,b^{-1}x}$ and $V_a^*S^{*m}e_{\oldz,x}$ has the form $ge_{t,a^{-1}x}$, so the inner product in the $(\oldz,x)$-summand is zero unless $a^{-1}x=b^{-1}x$ in $\nx$, or equivalently, unless $a=b$ and $a\mid x$. Similarly, since $S^{*n}e_{\oldz,x}$ has the form $he_{s,x}$, it is either in the range of $V_aV_a^*$ or orthogonal to it.
Thus
\begin{align*}
\psi(s^m v_a v_a^* s^{*n})
&=\frac{1}{\zeta(\beta-1)}\sum_{\{(\oldz,x)\in X\;:\;a|x\}}x^{-\beta}\langle S^m V_a V_a^* S^{*n}e_{\oldz,x},e_{\oldz,x}\rangle\\
&=\frac{1}{\zeta(\beta-1)}\sum_{\{(\oldz,x)\in X\;:\;a|x,\ S^{*n}e_{\oldz,x}\in V_aV_a^*(H_\mu)\}}x^{-\beta}\langle S^mS^{*n}e_{\oldz,x},e_{\oldz,x}\rangle.
\end{align*}
For each $x$ such that $a|x$, there are precisely $a^{-1}x$ elements $\oldz \in \Z/x$ such that $S^{*n}e_{\oldz,x}$ belongs to the range of $V_aV_a^*$. For each such element, 
\[
\langle S^mS^{*n}e_{\oldz,x},e_{\oldz,x}\rangle=
\begin{cases}0&\text{ unless $x$ divides $m-n$}\\
\int_{\T}\newz^{(m-n)/x}\,d\mu(\newz)&\text{ if $x|(m-n)$.}
\end{cases}
\]
So the right-hand side of~\eqref{KMScomputation} vanishes unless $a=b$ and $a|(m-n)$ (which ensures that there exist $x$ satisfying $a|x|(m-n)$), and in that case
\begin{equation}\label{formulaforpsi}
\psi(s^m v_a v_a^* s^{*n})
=\frac{1}{\zeta(\beta-1)}\sum_{\{x\in\nx\;:\;a|x,\ x|(m-n)\}}(a^{-1}x)x^{-\beta}\int z^{(m-n)/x}\,d\mu (\newz),
\end{equation}
as required.

To check that $\psi$ is a KMS$_{\beta}$ state using \lemref{KMScharacterisationlemma}, we need to see that when $a|(m-n)$ we have
\[
\psi(s^m v_a v_a^* s^{*n})=a^{-\beta}\psi(s^{(((m-n)/a))}),
\]
where we use double parentheses according to Convention~\ref{convention(())}.
However, expanding out the right-hand side gives
\begin{align*}
a^{-\beta}\psi(s^{((m-n)/a))})
&=\frac{a^{-\beta}}{\zeta(\beta-1)}\sum_{(w,y)\in X}y^{-\beta}\langle S^{(((m-n)/a))}e_{w,y},e_{w,y}\rangle\\
&=\frac{a^{-\beta}}{\zeta(\beta-1)}\sum_{\{(w,y)\in X\;:\;y|(m-n)/a\}}y^{-\beta}\Big(y\int_{\T}\newz^{(m-n)/ay}\,d\mu(\newz)\Big),
\end{align*}
which is another way of writing the right-hand side of~\eqref{formulaforpsi}.

For $\beta = \infty$, we use the same representation $\pi_\mu$, and set
\[
\psi_{\infty, \mu} (T) := \langle \pi_\mu( T) e_{0,1}, e_{0,1} \rangle;
\]
the state $\psi_{\infty,\mu}$ satisfies  \eqref{KMSinftycomputation}, and is a ground state by  \lemref{groundcharacterisation}.
To see that $\psi_{\infty,\mu}$ is a KMS$_\infty$ state,
notice first  that the only term on the right of \eqref{KMScomputation}
which survives the limit as $\beta \to \infty$ has $x = 1$, and there is such a term only when $a = 1$.
Since $\lim_{\beta \to \infty} \zeta(\beta-1) = 1$, we deduce that $\psi_{\beta,\mu}$ converges weak* to $\psi_{\infty,\mu}$ as $\beta\to\infty$, and  hence $\psi_{\infty,\mu}$ is a KMS$_\infty$ state.
Formula~\eqref{KMScomputation}, or formula \eqref{KMSinftycomputation} for $\beta = \infty$, shows that
the map $\mu\mapsto \psi_{\beta,\mu}$
is affine and weak*-continuous from $P(\T)$ into the simplex of KMS$_{\beta}$ states.

We claim that the states $\psi_{\infty, \newz}$ for $ \newz \in \T$ are pure and mutually inequivalent. The vector $e_{0,1} \in H_\newz = \ell^2(X)$ is cyclic for $\pi_\newz$ for each $\newz \in \T$, and thus $\pi_\newz$ can be regarded as the GNS representation of the corresponding vector state
\[
\psi_{\infty,\newz} (T) := \langle \pi_\newz(T)e_{0,1},e_{0,1}\rangle.
\]
So it suffices to show that if $A \in \mathcal B( \ell^2(X))$ is a nonzero projection intertwining $\pi_\newz$ and $\pi_w$ for some $\newz, w\in \T$, then $\newz = w$ and $A= 1$. Before we do this, we observe  that the product 
 \[
Q: =\prod_{p\in \primes} \prod_{j = 0}^{p-1} \big(1 - \pi_z(s^j)V_p V_p^*\pi_z(s^{*j})\big),
 \]
converges in the weak-operator topology on $\ell^2(X)$ to the rank-one projection onto $\C e_{0,1}$.
Indeed, we have $V_p^*\pi_z(s^{*j})e_{0,1}= \bar\newz^j V_p^*e_{0,1}=0$ for every $(j,p)$, and hence for each finite subset $F$ of $\primes$ we have $(1 - \pi_z(s^j)V_p V_p^*\pi_z(s^{*j}))e_{0,1}=e_{0,1}$ for all $p\in F$ and $j<p$;
on the other hand, if $b \neq 1$, there are a prime $p$ that divides $b$ 
and a value of $j$ such that $(j,p) \leq (n,b)$ in the quasi-lattice order
(see the proof of \thmref{qnisboundaryquotient}), and then
 $(1 - \pi_z(s^j)V_p V_p^*\pi_z(s^{*j})) e_{n,b} = 0$.
More generally, for each $a\in \nx$ and $0\leq m < a$,  
\[
Q_{m,a}  :=  \pi_z(s^m)V_a  Q V_a^* \pi_z(s^{*m}),
\]
is the rank-one projection onto the vector $e_{m,a}$, and is in $\pi_\newz(\TT(\nxnx))''$ for every $z\in \T$. Notice that the operator $Q_{m,a}$ on $\ell^2(X)$ is the same for every $\newz$.

Suppose now that $A \in \mathcal B( \ell^2(X))$ is a projection intertwining $\pi_\newz$ and $\pi_w$ for some $\newz, w\in \T$. 
Since  $Q_{m,a}$ belongs to $\pi_\newz(\TT(\nxnx))''$ and $\pi_w(\TT(\nxnx))''$, it commutes with $A$. This implies that there are scalars $\lambda_{k,a}$ such that $Ae_{k,a} = \lambda_{k,a}e_{k,a}$. Then we have
\begin{align*}
  \newz^n \lambda_{k,a}e_{k,a} &= A(\newz^ne_{k,a}) = A (\pi_z(s^{na})e_{k,a})=A (\pi_z(s^{na+k})V_a e_{0,1})\\
  &=   \pi_w(s^{na+k}) V_a  (Ae_{0,1})=
 \pi_w(s^{na+k}) V_a(\lambda_{0,1}e_{0,1}) = \lambda_{0,1}w^ne_{k,a}.
\end{align*}
Thus
 $\lambda_{k,a} = (w/\newz)^n   \lambda_{0,1} $ for every $n\in \N$,
 and this implies that either  $\lambda_{k,a} = 0$ for every $(k,a)$, or
 $\newz =w$, in which case $\lambda_{k,a} = \lambda_{0,1}$ for every $(k,a)$. Either way, $A$ is a multiple of the identity, and the representations $\pi_\newz$ are irreducible and mutually inequivalent. Thus the corresponding vector states $\psi_{\infty,\newz}$ are pure
and mutually inequivalent.
\end{proof}
 
We now prove the parts of  \thmref{maintheorem} which describe the ground states on $\TT(\nxnx)$.

\begin{proof}[Proof of part (4) of \thmref{maintheorem}]
Since the additive generator $s$ of $\TT(\nxnx)$ is a proper isometry, Coburn's Theorem implies that $C^*(s)$ is naturally isomorphic to $\TT(\N)$. The restriction $\omega:=\psi|_{C^*(s)}$ is then a positive functional satisfying $\omega(1)=1$, and hence is a state of $C^*(s)\cong \TT(\N)$. So $\psi\mapsto \psi|_{C^*(s)}$ maps ground states to states of $\TT(\N)$. \lemref{lemmagroundcharacterisation} implies that $\psi$ satisfies \eqref{formground}, which implies that $\psi\mapsto \psi|_{C^*(s)}$ is injective  on ground states.

To see that $\psi\mapsto \psi|_{C^*(s)}$ is surjective, let $\omega$ be a state of $\TT(\N)$, let $\pi_{S,V}$ be the representation of $\TT(\nxnx)$ on $\ell^2(X,\HH_\omega)$ constructed in \lemref{constructreps}, and define
\[
\psi_{\omega}(T):=\langle \pi(T) \xi_\omega e_{0,1}, \xi_\omega e_{0,1}\rangle\ \text{ for $T\in \TT(\nxnx)$.}
\]
We then have
\begin{equation}\label{comppsiomega}
\psi_\omega(s^m v_a v_b^* s^{*n})= 
\langle S^m V_aV_b^*S^{*n} \xi_\omega e_{0,1}, \xi_\omega e_{0,1}\rangle = \langle V_b^*S^{*n}\xi_\omega e_{0,1}, V_a^*S^{*m}\xi_\omega e_{0,1}\rangle.
\end{equation}
Since $V_b^*S^{*n}\xi_\omega e_{0,1}$ vanishes unless $b=1$, the right-hand side of \eqref{comppsiomega} vanishes unless $a=b=1$, and \lemref{lemmagroundcharacterisation} implies that $\psi_\omega$ is a ground state. On the other hand, if $a=b=1$, then \eqref{comppsiomega} gives
\[
\psi_\omega(s^m v_a v_b^* s^{*n})= 
\langle S^m S^{*n} \xi_\omega e_{0,1}, \xi_\omega e_{0,1}\rangle
=\langle \pi_{\omega}(s^m s^{*n}) \xi_\omega, \xi_\omega\rangle=\omega(s^ms^{*n}),
\]
which implies that $\psi_\omega|_{C^*(s)}=\omega$. We now know that $\psi\mapsto \psi|_{C^*(s)}$ is a bijection from the set of ground states onto the state space of $\TT(N)$. 

The map $\psi\mapsto \psi|_{C^*(s)}$ is obviously affine. Equation~\eqref{formground} implies that it is a homeomorphism for the respective weak* topologies, and hence it is an affine isomorphism of compact convex sets. This implies in particular that the extremal ground states are those of the form $\psi_\omega$ where $\omega$ is a pure state of $\TT(\N)$. Since the GNS representation $\pi_\omega$ of $\TT(\N)$ is irreducible, the Wold decomposition for the isometry $\pi_\omega(s)$ implies that $\pi_\omega$ is either equivalent to the identity representation of $\TT(\N)$ on $\ell^2$ or is lifted from an irreducible representation of $C(\T)$. Thus, since $\omega$ is a vector state in its GNS representation, $\omega$ is either a vector state for the identity representation or is lifted from an evaluation map on $C(\T)$.
\end{proof}

\section{Surjectivity of the parametrisation of KMS states}\label{secsurjectivity}
To show that all KMS states 
arise via the above construction, we need to show that in any GNS representation, there are analogues of the projection $Q$ which we used in the proof of \proref{constructKMS>2}. To deal with the case where $\zeta(\beta-1)$ does not converge, we need to use also analogous projections involving products over finite sets of primes. For each subset $\setb$ of $ \primes$, let $\nx_\setb$ denote the 
semigroup of positive integers with all prime factors in $\setb$; 
the corresponding zeta function and Euler product are given by
\begin{equation}\label{eulerproductB}
\zeta_\setb(\beta) := \sum_{a\in \nx_\setb} a^{-\beta}= \prod_{p\in \setb} (1-p^{-\beta})\inv.
\end{equation}
For every $\setb$, the series converges for $\beta>1$, but if $\setb$ is finite it also converges 
for $\beta > 0$.

The reconstruction formula in part (3) of the following lemma is one of our main technical innovations. We will see in the Appendix how this technique also simplifies the proof of uniqueness for the KMS states of the Bost-Connes systems, and we believe that it is likely to be useful elsewhere.

\begin{lemma}\label{phirestrictedtoQB}
Let $\beta >1$ and suppose $\phi$ is a KMS$_{\beta}$ state. 
Form the GNS-representation $(H_\phi,\pi_\phi,\xi_\phi)$ of $\TT(\nxnx)$, so that $\phi(\cdot)=\langle \pi_\phi(\cdot)\xi_\phi,\xi_\phi\rangle$, and 
denote by $\tilde\phi$ the vector state extending $\phi$ to all bounded operators on $H_\phi$.  Write $S=\pi_\phi(s)$, $V_p=\pi_\phi(v_p)$, and let $E$ be a subset of $\primes$.  Then the product 
\[
 Q_\setb: =\prod_{p\in \setb} \prod_{j = 0}^{p-1} (1 - S^j V_p V_p^*S^{*j})
 \]
converges in the weak-operator topology to a projection $Q_\setb$ in $\pi_\phi(\TT(\nxnx))''$, which satisfies 
 \begin{enumerate}
 \item  $\tilde\phi(Q_\setb) ={ \zeta_\setb(\beta-1)}\inv$;
 \smallskip
 \item  if $E$ is a subset of $\primes$ such that $\zeta_\setb(\beta -1)  < \infty$, then $\phi_{Q_\setb} (T) := \zeta_\setb(\beta -1) \tilde \phi(Q_\setb \pi_\phi(T) Q_\setb)$ defines a state $\phi_{Q_\setb}$ of  $\TT(\nxnx)$, called the \emph{conditional state of $\phi$ with respect to $Q_\setb$};
  \smallskip
 \item if $\zeta_\setb(\beta -1)  < \infty$, then $\phi$ can be reconstructed from  its conditional state $\phi_{Q_\setb}$ by the formula
\begin{equation}\label{reconstruct}
 \phi (T) =  \sum_{a\in \nx_\setb}\sum_{k=0}^{a-1} \frac{a^{-\beta}}{\zeta_\setb(\beta-1)} \phi_{Q_\setb}(v_a^* s^{*k} \, T \, s^k v_a);
\end{equation}
in particular, for $n \geq 0$ we have
\begin{equation}\label{reconstructionequationSn}
\phi(s^n)=\frac{1}{\zeta_\setb(\beta-1)}\sum_{\{a\in \nx_\setb\;:\;a|n\}}a^{1-\beta}\phi_{Q_\setb}(s^{n/a}).
\end{equation}
 \end{enumerate}
\end{lemma}

\begin{proof}
When $\setb$ is finite, the product is finite and belongs to $\pi_\phi(\TT(\nxnx))$. When $\setb$ is infinite, $Q_\setb$ is the weak-operator limit of a decreasing family of projections in the range of $\pi_\phi$, and  therefore belongs to $\pi_\phi(\TT(\nxnx)) ''$.

Suppose $p$ and $q$ are relatively prime.  Since for each $a$ the projections $s^j v_av_a^* s^{*j}$ with $0\leq j < a$ 
have mutually orthogonal ranges, we have
\begin{align*}
\phi\Big(\prod_j (1 - s^j v_pv_p^* s^{*j}) \, &\prod_k (1 - s^k v_qv_q^* s^{*k}) \Big) 
= \phi\Big(\Big(1 - \sum_j s^j v_pv_p^* s^{*j}\Big)\Big(1- \sum_k s^k v_qv_q^* s^{*k}\Big)\Big)\\
&= \phi\Big(1 - \sum_j s^j v_pv_p^* s^{*j}- \sum_k s^k v_qv_q^* s^{*k}+\sum_{j,k}s^j v_pv_p^* s^{*j}\;s^k v_qv_q^* s^{*k}\Big).
\end{align*}
The covariance relation in Lemma~\ref{covarianceonngenerators} implies that
\[
v_p^*s^{*j}s^kv_q=s^\alpha v_qv_p^*s^{*\beta},
\]
where $(\alpha,\beta)$ is the smallest non-negative solution of $k-j=\alpha p-\beta q$. So
\begin{align*}
\phi\Big(\prod_j (1 - s^j v_pv_p^* s^{*j}) \, &\prod_k (1 - s^k v_qv_q^* s^{*k}) \Big) \\
&= 1 - \sum_j \phi(s^j v_pv_p^* s^{*j}) - \sum_k \phi(s^k v_qv_q^* s^{*k})  + \sum_{j,k}
\phi( s^j v_pv_p^* s^{*j}\;s^kv_qv_q^* s^{*k} )\\
&= 1 - \sum_j p^{-\beta} - \sum_k  q^{-\beta}  + \sum_{j,k}
\phi\big( s^{j+p\alpha}v_pv_qv_q^*v_p^* s^{*(k+q\beta)}\big)\\
&= 1 - p(p^{-\beta}) - q(q^{-\beta})  + \sum_{j,k} (pq)^{-\beta}\\
&= (1-p^{1-\beta}) \, (1-q^{1-\beta})\\
&=\phi\Big(\prod_j (1 - s^j v_pv_p^* s^{*j})\Big)\phi\Big(\prod_k (1 - s^k v_qv_q^* s^{*k})\Big).
\end{align*}
where we have used formula \eqref{KMScharacterisation} in the third equality.
From this we deduce that for every finite subset $F$ of $\setb$, we have 
\[
\phi\Big(\prod_{p\in F} \prod_{j=0}^{p-1} (1-s^jv_{p} v_{p}^* s^{*j}) \Big)
= \prod_{p\in F} (1-p^{1-\beta}),
\]
and (1) follows on taking limits and using the product formula
\eqref{eulerproductB} for $\zeta_\setb$. 

Since $A \mapsto Q_\setb AQ_\setb$ is  positive and linear, $\phi_{Q_\setb }$ is a positive linear functional; part (1) implies that $\phi_{Q_\setb }(1) = 1$, and we have proved (2).

We next claim  that the projections 
\[
\{Q_{\setb ,k,a}:= S^k V_a Q_\setb  V_a^* S^{*k}: a\in \nx_E,\ 0\leq k<a\}
\]
are mutually orthogonal. Suppose $a,b \in \nx_\setb $, $0 \leq k <a$ and $0\leq l < b$ satisfy $(k,a)\not=(l, b)$. Then
\begin{equation}\label{prodorthog}
Q_{\setb ,k,a} Q_{\setb ,l,b} = S^k V_a Q_\setb  (V_a^* S^{*k} S^l V_b)  Q_\setb  V_b^* S^{*l},
\end{equation}  
and the covariance relation of Lemma~\ref{covarianceonngenerators} implies that the factor in parenthesis has the form $S^\gamma V_{b'} V_{a'}^* S^{*\delta}$, where $(\gamma, b')$ and $(\delta, a') $ cannot both be equal to $(0,1)$ because $(k,a)\not=(l, b)$. We can now use (T1) to extract from either $S^\gamma V_{b'}$ or $S^{\delta}V_{a'}$  a factor 
of the form $S^kV_p$ with $p \in \setb $ and $0\leq k<p$; since $Q_\setb  \leq 1 - S^k V_p V_p^* S^{*k}$, this implies that the right-hand side of \eqref{prodorthog} vanishes, and the claim is proved.

If $\zeta_\setb (\beta -1) < \infty$, then 
\[
\tilde\phi(\sum_{k,a} Q_{\setb ,k,a}) = \sum_{k,a} a^{-\beta} \tilde\phi(Q_\setb ) = \sum_a aa^{-\beta}\tilde\phi(Q_\setb )= \tilde\phi(Q_\setb ) \zeta_\setb (\beta -1) = 1,
\]
so that $\tilde \phi$ is carried by the projection $\sum_{k,a} Q_{\setb ,k,a}$. Thus we have
\[
\phi(T) 
=\tilde\phi\Big(\Big(\sum_{k,a}Q_{\setb ,k,a}\Big)\pi_\phi(T) \Big(\sum_{l,b}Q_{\setb ,l,b}\Big)\Big)
=\sum_{k,a, l,b} \tilde\phi\big(Q_{\setb ,k,a} \pi_\phi(T)  Q_{\setb ,l,b}\big).
\]
Now the orthogonality of the projections $Q_{E,k,a}$ and the KMS$_\beta$ condition imply that 
  \[
\phi(T)  = \sum_{k,l,a,b}a^{-\beta} \tilde\phi \big(Q_\setb  V_a^* S^{*k} \pi_\phi(T) S^l V_b Q_\setb V_b^*S^{*l}S^kV_aQ_\setb\big);
\]
as in \eqref{prodorthog}, we have $Q_\setb V_b^*S^{*l}S^kV_aQ_\setb=0$ unless $(k,a)=(l,b)$, and hence
  \[
\phi(T)  = \sum_{k=l, \, a=b}a^{-\beta} \tilde\phi \big(Q_\setb  V_a^* S^{*k} \pi_\phi(T) S^l V_b Q_\setb \big),
\]
which implies the reconstruction formula \eqref{reconstruct}. To get the formula \eqref{reconstructionequationSn} for $\phi(s^n)$, we deduce from the formulas in Lemma~\ref{relsata} that
\[
v_a^* s^{*m}s^ns^m v_a =
\begin{cases} v_a^*s^nv_a=s^{n/a}&\text{ if $a|n$,}\\ 
0&\text{ otherwise,}
\end{cases}
\]
so that for each $a|n$ there are $a$ equal summands on the right-hand side of \eqref{reconstruct}. This completes the proof of part (3).
\end{proof}

\begin{proposition}[\thmref{maintheorem}(2): uniqueness for $1\leq \beta \leq 2$]\label{unique1-2}
The state $\psi_\beta$ constructed in \proref{productmeasure} is the unique 
KMS$_\beta$ state for $1 \leq \beta \leq 2$.
\end{proposition}

Before proving Proposition~\ref{unique1-2} we need to do some preliminary work. 

\begin{lemma}\label{checkprojOK}
Suppose $\beta\geq 1$ and $\phi$ is a KMS$_{\beta}$ state of $\TT(\nxnx)$. If $P$ is a projection in the span of $\{s^mv_av^*_bS^{*n}\}$ such that $\sigma_t(P)=P$ for all $t\in \R$ and $\phi(P)=0$, then $\phi(RPT)=0$ for all $R,T\in \TT(\nxnx)$.
\end{lemma}

\begin{proof}
We first observe that for every  $T\in \TT(\nxnx)$ we have
\[
0\leq \phi(PT^*TP)\leq \phi(P\|T\|^2P)=\|T\|^2\phi(P)=0,
\]
and hence $\phi$ vanishes on the corner $P \TT(\nxnx) P$. Next, we consider analytic elements $R=s^mv_av_b^*s^{*n}$ and $T=s^qv_cv_d^*s^{*r}$. Since $z\mapsto P-\sigma_{z}(P)$ is analytic and vanishes on $\R$, it vanishes everywhere.  Thus the KMS$_{\beta}$ condition gives
\[
\phi(RPT)=\phi((RP)(PT))=\phi(PT\sigma_{i\beta}(RP))=(a/b)^{-\beta}\phi(PRTP)=0,
\]
and this extends to arbitrary $R$ and $T$ by continuity of $\phi$. 
\end{proof}

\begin{lemma}\label{KMSfactors}
Suppose that $\phi$ is a KMS$_{\beta}$ state of $\TT(\nxnx)$ for some $\beta\geq 1$. Then $\phi$ vanishes on the ideal in $\TT(\nxnx)$ generated by $1-ss^*$. If $\beta=1$, then $\phi$ also vanishes on the ideal generated by $\{1-\sum_{k=0}^{p-1} s^jv_pv_p^*s^{*j}:p\in \primes\}$.
\end{lemma}

\begin{proof}
From \eqref{KMScharacterisation} we have $\phi(1-ss^*)=\phi(1)-\phi(ss^*)=1-1=0$, so the first assertion follows from Lemma~\ref{checkprojOK}. Now suppose $\beta=1$. Then another application of \eqref{KMScharacterisation} shows that
\[
\phi\Big(1-\sum_{k=0}^{p-1} s^jv_pv_p^*s^{*j}\Big)=1-\sum_{k=0}^{p-1}p^{-1}=1-p(p^{-1})=0,
\]
so the second assertion also follows from Lemma~\ref{checkprojOK}.
\end{proof}

\begin{proof}[Proof of Proposition~\ref{unique1-2}]
Let $\phi$ be a KMS$_\beta$ state, and 
suppose first that $1< \beta \leq 2$. For every finite set $E\subset\primes$ and every $n > 0$, the sum in \eqref{reconstructionequationSn}
has finitely many summands, each satisfying $a^{1-\beta}|\phi_{Q_\setb } (s^{n/a}) | \leq a^{1-\beta}$.  Thus, since $\zeta_\setb (\beta-1) \to \infty$ as $\setb $ increases,
the right-hand side of \eqref{reconstructionequationSn} tends to zero as $\setb $ increases through a listing of $\primes$. Thus $\phi(s^n) = 0$ for every $n\in \Z\setminus \{0\}$, and \lemref{KMScharacterisationlemma} implies that 
\[
\phi(s^nv_bv_a^*s^{*m})=\begin{cases}
0&\text{ unless $a=b$ and $m=n$}\\
a^{-\beta}&\text{ if $a=b$ and $m=n$.}
\end{cases}
\]
Comparing this with \eqref{charpsibeta} shows that $\phi = \psi_\beta$.

Now suppose $\beta = 1$.  Then Lemma~\ref{KMSfactors}  implies that $\phi$ 
 factors through the boundary quotient, and thus comes from a state of Cuntz's $\qn$. Thus the result follows from \cite[Theorem 4.3]{cun2}.
\end{proof}

For $\beta >2$ we can take $\setb  = \primes$ in Lemma~\ref{phirestrictedtoQB}, and deduce that a KMS$_\beta$ state is determined by its conditioning to  $Q:=Q_\primes$. We shall use this to prove that the map described in part (3) of \thmref{maintheorem} is surjective. 
Since KMS states vanish on the ideal generated by $1-ss^*$, the state $\phi$, the GNS-representation $\pi_\phi$, and the conditional state $\phi_{Q}$ all vanish on that ideal. In particular, this implies that the restriction of $\phi_{Q}$
 to $C^*(s)$ factors through the quotient map $q:C^*(s)\to C(\T)$, and hence there is a probability measure $\mu = \mu_\phi$ on $\T$ such that 
\begin{equation}\label{phiQ}
\phi_{Q}(s^n)=\int_{\T}q(s^n)\,d\mu=\int_{\T} \newz^n\,d\mu(\newz)\ \text{ for $n\in\Z$}.
\end{equation}

\begin{proposition}[\thmref{maintheorem}(3): the map $\mu \to \psi_{\beta,\mu}$
is a bijection] \label{KMSfromrestriction}
Let $\beta>2$ and take $Q:= Q_\primes$. If $\phi$ is a KMS$_\beta$ state and
$\mu_\phi$ is the probability measure on $\T$ such that \eqref{phiQ} holds, then $\phi = \psi_{\beta,\mu_\phi}$. Conversely, if $\mu$ is a probability measure on $\T$, then $\mu = \mu_{\psi_{\beta, \mu}}$.
\end{proposition}

\begin{proof}
By \lemref{KMScharacterisationlemma}, to prove the first assertion
it suffices to check that $\phi$ and $ \psi_{\beta,\mu}$ agree on positive powers of $s$ (taking adjoints then shows that they also agree on powers of $s^*$). Since $\zeta(\beta -1) <\infty$, the reconstruction formula  
\eqref{reconstructionequationSn} gives
\begin{align*}
\phi(s^n)&=\frac{1}{\zeta(\beta-1)}\sum_{\{a\in \nx\;:\;a|n\}}a^{1-\beta}\phi_Q(s^{n/a})\\
&=\frac{1}{\zeta(\beta-1)}\sum_{\{a\in \nx\;:\;a|n\}} a^{1-\beta}\int_{\T} \newz^{n/a}\,d\mu(\newz),
\end{align*}
which by \eqref{KMScomputation} is precisely $\psi_{\beta,\mu}(s^n)$.

For the converse, we show that the moment sequences  
  $\int_\T \newz^m \, d\mu(\newz) $ and $\int_\T \newz^m \, d \mu_{\psi_{\beta, \mu}}(\newz)$
for the two measures coincide, and then an application of the Riesz representation theorem shows that $\mu=\mu_{\psi_{\beta, \mu}}$.  Since the measures are positive it suffices to deal with 
$m\geq 0$. We know that \eqref{KMScomputation} holds for both measures, for
$\mu$ by definition and for $ \mu_{\psi_{\beta, \mu}}$ by the first part.
When $m = 1$, we take $n = 0$  and  $x = y = 1$ in \eqref{KMScomputation}, which then reduces to a single term, and we can deduce that the first moments coincide:
 \[
 \int_\T \newz \, d\mu(\newz)  = \zeta(\beta-1) {\psi_{\beta, \mu}}(s) =\int_\T \newz \, d \mu_{\psi_{\beta, \mu}}(\newz).
 \]
 When  $m= p \in \primes$, equation \eqref{KMScomputation} gives
\begin{align}
  \int_\T \newz^p \, d\mu(\newz) + p^{1-\beta} \int_\T \newz \, d\mu(\newz)& = \zeta(\beta-1) {\psi_{\beta, \mu}}(s^p)\label{momentcalc}\\
  &=\int_\T \newz^p \, d \mu_{\psi_{\beta, \mu}}(\newz) + p^{1-\beta} \int_\T \newz \, d \mu_{\psi_{\beta, \mu}}(\newz).\notag
\end{align} 
Since we already know that $\int_\T \newz \, d\mu(\newz) = \int_\T \newz \, d \mu_{\psi_{\beta, \mu}}(\newz)$,
we conclude that 
\[
 \int_\T \newz^p \, d\mu(\newz)  = \int_\T \newz^p \, d \mu_{\psi_{\beta, \mu}}(\newz).
 \]
We can extend this result to non-prime $n\in \nx$ by an induction argument on the number of prime factors of $m$ (counting multiplicity); the key inductive step is established by an argument like that of \eqref{momentcalc}. Thus the moments are equal for all $n$, and the result follows.
 \end{proof}

This concludes the proof of  \thmref{maintheorem}.

\appendix
 \section{Uniqueness of equilibrium for the  Bost-Connes algebra}\label{bcuniqueness}
 
The Hecke $C^*$-algebra $\bcheck$ of Bost and Connes \cite{bos-con} is the universal unital $C^*$-algebra generated by a unitary representation $e:\Q/\Z\to U(\bcheck)$ and an isometric representation $\mu:\nx\to \bcheck$ satisfying
\[
\frac{1}{n}\sum_{\{s\in\Q/\Z\,:\,ns=r\}} e(s)=\mu_ne(r)\mu_n^*
\]
(see \cite[Corollary~2.10]{bcalg}). The isometric representation $\mu$ is then automatically Nica covariant \cite[Proposition~2.8]{bcalg}, so there is a natural homomorphism $\pi_\mu:\TT(\nx)\to \bcheck$, and the main theorem of \cite{quasilat} implies that $\pi_\mu$ is injective. The unitary representation $e$ induces a unital homomorphism $\pi_e:C^*(\Q/\Z)\to \bcheck$. Since $\Q/\Z$ is an abelian group with dual isomorphic  to the additive group $\widehat\Z$ of integral ad\`eles, we can view $\pi_e$ as a homomorphism $\pi:C(\widehat\Z)\to\bcheck$. There is an action $\alpha$ of $\nx$ on $C(\widehat \Z)$ defined by
\[
\alpha_n(f)(z)=\begin{cases}f(n^{-1}z)&\text{if $n$ divides $z$ in $\widehat\Z$}\\
0&\text{otherwise,}\end{cases}
\]
and then the relations defining $\bcheck$ say that $(\pi,\mu)$ satisfies
\begin{equation}\label{semigpcov}
\pi(\alpha_n(f))=\mu_n\pi(f)\mu_n^*,
\end{equation}
and that $(\pi,\mu)$ is universal for such pairs (see, for example, \cite[Proposition~32]{diri}).  Next, note that the endomorphism $\gamma_n$ of $C(\hatz)$ defined by $\gamma_n(f)(z)=f(nz)$ satisfies $\alpha_n\circ\gamma_n(f)=\alpha_n(1)f$, so the embedding $\pi:C(\hatz)\to\bcheck$ satisfies
\begin{align}\label{revcov}
\mu_n^*\pi(f)\mu_n&=\mu_n^*\mu_n\mu_n^*\pi(f)\mu_n=\mu_n^*\pi(\alpha_n(1)f)\mu_n=\mu_n^*\pi(\alpha_n\circ\gamma_n(f))\mu_n\\
&=\mu_n^*\big(\mu_n\pi(\gamma_n(f))\mu_n^*\big)\mu_n=\pi(\gamma_n(f)).\notag
\end{align}

\begin{example}
In Theorem~\ref{qnisboundaryquotient} we identified $\qn$ as the boundary quotient $C(\partial\Omega)\rtimes(\qxqx)$. Proposition~\ref{boundary} shows that the homeomorphism of $\widehat\Z$ onto $\partial\Omega$ carries the action of $\nx\subset \widehat\Z$ by left multiplication (in the ring $\widehat\Z$) into the left action of $\nx\subset \Q_+^*\subset \qxqx$ on $\partial\Omega$. Thus if we use this homeomorphism to define a homomorphism $\pi:C(\widehat\Z)\to C(\partial\Omega)\rtimes(\qxqx)$, then the pair $(\pi,v|_{\nx})$ satisfies \eqref{semigpcov}, and hence gives a homomorphism $\pi\times v|_{\nx}$ of $\bcheck$ into $\qn$. Theorem~3.7 of \cite{bcalg} implies that $\pi\times v|_{\nx}$ is injective. (Cuntz gave a slightly different description of this embedding in \cite[Remark~3.5]{cun2}.)
\end{example}

The universal property of $\bcheck$ implies that there is an action $\sigma:\R\to \Aut\bcheck$ which fixes the subalgebra $C^*(\Q/\Z)\cong C(\widehat\Z)$ and satisfies $\sigma_t(\mu_n)=n^{it}\mu_n$. Our goal in this short appendix is to use the ideas of \secref{secsurjectivity} to give a relatively elementary proof of the following theorem, which is a key part of the Bost-Connes analysis of $\bcheck$. This approach bypasses the technical proofs 
of \cite[Lemmas~27(b) and~28]{bos-con} and of \cite[Lemma~45]{diri}.

\begin{theorem}\label{BCbeta<1}
For $\beta\in (0,1]$, the system $(\bcheck,\R,\sigma)$ has at most one KMS$_\beta$ state.
\end{theorem}

As in \cite{bos-con}, the idea is to prove that a KMS$_\beta$ state is invariant under the action of a large symmetry group as well as the dynamics. When we view $\bcheck$ as being generated by $C(\widehat\Z)$ and an isometric representation $\mu$ of $\nx$, the symmetry group is the multiplicative group $\widehat\Z^*$ of invertible elements in the ring $\widehat\Z$, which acts on $C(\widehat \Z)$ by $\tau_u(f)(z)=f(uz)$. The automorphisms $\tau_u$  commute with the endomorphisms $\alpha_n$, and hence give an action $\theta:\widehat\Z^*\to \Aut\bcheck$. This action commutes with the dynamics $\sigma$, and we know from Proposition~21 of \cite{bos-con} (or Propositions 30 and 32 of \cite{diri}) that the fixed-point algebra $\bcheck^\theta$ is the copy of the Toeplitz algebra $\TT(\nx)$ in $\bcheck$.

The key lemma is:

\begin{lemma}\label{KMSinv} 
Any KMS$_\beta$ state for $0< \beta \leq 1$ is $\hatz^*$-invariant, cf. \cite[Lemma 27(c)]{bos-con}.
\end{lemma}
 
Given this lemma, we know that any KMS$_\beta$ state $\psi$ factors through the expectation $E_\theta$ onto the fixed point algebra $\bcheck^\theta=\TT(\nx)$, and hence is determined by its values on $\TT(\nx)$. Since $\nx$ is quasi-lattice ordered (in fact it is lattice ordered), we have $\TT(\nx)=\clsp\{\mu_m\mu_n^*\}$. Since the KMS state $\psi$ is invariant for the dynamics $\sigma$, and since $\sigma_t(\mu_m\mu_n^*)=(m/n)^{it}\mu_m\mu_n^*$, we must have $\psi(\mu_m\mu_n^*)=0$ for $m\not= n$, and $\psi$ is determined by its values on $\lsp\{\mu_n\mu_n^*\}$. But there it is completely determined by the KMS condition:
\[
\psi(\mu_n\mu_n^*)=\psi(\mu_n^*\sigma_{i\beta}(\mu_n))=n^{-\beta}\psi(\mu_n^*\mu_n)=n^{-\beta}\psi(1)=n^{-\beta}.
\]
So there can only be one such state, and Lemma~\ref{KMSinv} implies Theorem~\ref{BCbeta<1}.

It remains for us to prove \lemref{KMSinv}. The key ingredient is an analogue of
\lemref{phirestrictedtoQB} for the Bost-Connes system, which stems from the observation that the reconstruction formula from \cite[Theorem 20]{diri} 
also works for small $\beta$ if one restricts to finitely many primes, as Neshveyev did in the proof of the Proposition in \cite{nes}.

Suppose that  $\phi$ is a KMS$_\beta$ state of $(\cq, \sigma)$. For $E\subset \primes$ finite, we set $Q_\setb  := \prod_{p\in \setb } (1- \mu_p \mu_p^*)$. For distinct primes $p,q$ we have
 \[
  \phi(\mu_p  \mu_p^* \mu_q \mu_q^*)  = \phi(\mu_{pq}\mu_{pq}^*) = (pq)^{-\beta} = 
 p^{-\beta} q^{-\beta} = \phi(\mu_p \mu_p^*) \phi( \mu_q \mu_q^*) ,
 \]
and thus
  \[
  \phi(Q_\setb ) = \prod_{p\in \setb } (1 - \phi(\mu_p\mu_p^*)) = \prod_{p\in \setb } (1 - p^{-\beta}) = \zeta_\setb (\beta)\inv,
  \]
as defined in \eqref{eulerproductB}.  We define the  conditional state $\phi_{Q_\setb }$ ($\phi$ given $Q_\setb $) by
 \[
 \phi_{Q_\setb }(\cdot) := \zeta_\setb (\beta) \phi(Q_\setb  \, \cdot \, Q_\setb ).
 \]

 \begin{lemma}\label{reconstructionlemma}
 If $\phi$ is a KMS$_\beta$ state of $(\cq,\sigma)$ and $E$ is a finite subset of $\primes$, then
 \[
 \phi(T) = \sum_{n\in \nx_\setb } \frac{n^{-\beta}}{\zeta_\setb (\beta)} \phi_{Q_\setb } (\mu_n^* T  \mu_n)\ \text{ for $T\in \cq$.}
 \]
 \end{lemma}
 \begin{proof}
We first claim that the projections $\mu_n Q_\setb  \mu_n^*$ for  $n\in \nx_\setb $ are mutually orthogonal. To see this, suppose $m,n\in \nx_\setb$ and $m\not=n$. Then there exists $q\in E$ such that $e_q(m)\not= e_q(n)$; say we have $e_q(m)<e_q(n)$, and write $m=m'q^{e_q(m)}$, $n=n'q^{e_q(n)}$. Then $\gcd(q,m')=1$, so  $\mu_{m'}^*\mu_q= \mu_q\mu_{m'}^*$ and 
\begin{align*} 
( \mu_m Q_\setb  \mu_m^*)  ( \mu_n Q_\setb  \mu_n^*)
&= \mu_m Q_\setb  \mu_{m'}^* \mu_q^{e_q(n)-e_q(m)}\mu_{n'} Q_\setb  \mu_n^*\\
&=\mu_m Q_\setb \mu_q^{e_q(n)-e_q(m)} \mu_{m'}^*\mu_{n'} Q_\setb  \mu_n^* 
\end{align*}
vanishes because the factor $(1-\mu_q\mu_q^*)\mu_q$ of $Q_\setb \mu_q^{e_q(n)-e_q(m)}$ does.

The normal extension $\tilde\phi$ of $\phi$ to $\pi_\phi(\cq)"$ satisfies 
\[
\tilde\phi \Big(\sum_{n \in \nx_\setb } \pi_\phi\big(\mu_n Q_\setb  \mu_n^*\big)\Big) = \sum_{n \in \nx_\setb } \phi(\mu_n Q_\setb  \mu_n^*)
= \sum_{n \in \nx_\setb } n^{-\beta} \phi(Q_\setb ) = 1,
\]
so the state $\tilde\phi$ is supported by the  projection $\sum_{m \in \nx_\setb } \mu_m Q_\setb  \mu_m^*$, and
 \begin{align*}
 \phi(T) &=\tilde\phi(\pi_\phi(T))\\
 &=\tilde  \phi\bigg(\Big(\sum_{m \in \nx_\setb }\pi_\phi\big( \mu_m Q_\setb  \mu_m^*\big)\Big) \pi_\phi(T)\Big(\sum_{n \in \nx_\setb } \pi_\phi\big(\mu_n Q_\setb  \mu_n^*\big)\Big)\bigg)\\
&= \sum_{m,n \in \nx_\setb } \phi(\mu_m Q_\setb  \mu_m^* T \mu_n Q_\setb  \mu_n^*).
 \end{align*}
The KMS condition \eqref{defKMS} (with $c  = \mu_m Q_\setb  \mu_m^*$) and the orthogonality of the $\{\mu_n Q_\setb  \mu_n^*\}$ imply that the terms with $m \not=n$ vanish, and another application of the KMS condition gives   
\[
 \phi(T) =\sum_{n \in \nx_\setb } n^{-\beta} \phi( Q_\setb  \mu_n^* T \mu_n Q_\setb  \mu_n^*n^{-\beta}\mu_n) = \sum_{n \in \nx_\setb } n^{-\beta} \phi( Q_\setb  \mu_n^* T \mu_n Q_\setb  ).\qedhere
 \]
\end{proof}

\begin{proof}[Proof of \lemref{KMSinv}]
As indicated at the beginning of \cite[Section 7]{bos-con}, to prove that a state is $\hatz^*$-invariant it suffices to show that
it vanishes on the spectral subspaces 
\[
C(\hatz)_\chi : = \{f \in C(\hatz) : \theta_g ( f ) = \chi(g) f \text{ for all } g \in \hatz^*\}
\]
for nontrivial characters  $\chi $ of $\hatz^*$. So suppose that\footnote{Sorry about the hats. The one in $\hatz^*$ is the hat in $\hatz$, which is standard number-theoretic notation for the integral ad\`eles, and the outside one in $(\hatz^*)^\wedge$ is the standard harmonic-analytic notation for the Pontryagin dual.} $\chi\in(\hatz^*)^\wedge$ and $\chi\not=1$. Since $\hatz^*$ is the inverse limit
$\varprojlim(\Z/n\Z)^*$, the dual $(\hatz^*)^\wedge$ is the direct limit $\varinjlim((\Z/n\Z)^*)^\wedge$, and there exists $m$ such that $\chi$ belongs to $((\Z/m\Z)^*)^\wedge$ --- in other words, such that $\chi$ factors through the canonical map $r\mapsto r(m)$ from $\hatz^*$ to $(Z/m\Z)^*$. Let $F$ be a finite set of primes containing all the prime factors of $m$, so that $m\in\nx_F$. When we identify $\hatz$ with $\prod_{p\in\primes}\Z_p$, the subalgebras $C_F:=C(\prod_{p\in F}\Z_p)\otimes 1$ span a dense subspace of $C(\hatz)$; since $E_\chi:f\mapsto \int_{\hatz^*} \theta_u(f)\overline{\chi(u)}\,du$ onto $C(\hatz)_\chi$ is continuous, the union $\bigcup_{F}(C_F\cap C(\hatz)_\chi)$ is dense in $C(\hatz)_\chi$. Thus it suffices to prove that $\phi(\pi(f))=0$ for every $f\in C_F$ (where $\pi$ is the embedding of $C(\hatz)$ in $\bcheck$ discussed at the beginning of the section). 

As in the proof of \cite[Lemma 27]{bos-con}, see also \cite[pages~369--370]{diri}, we modify the embedding of $\nx$ in $\hatz$ so that every positive integer lands in $\hatz^*$: for $q\in \primes$ we take $u_q$ to be the element of $\prod_p \Z_p^*$ defined by 
\[
(u_q)_p = 
\begin{cases}q & \text{ if  $p \neq q$} \\
                  1 & \text{ if } p = q,
\end{cases}
\]
and extend the map $q\mapsto u_q$ to $\nx$   by prime factorisation. Notice that if $p\notdiv n$, then $(u_n)_p=n$ in $\Z_p$, so for functions $f\in C_F$ and $n\in \nx_{\primes\setminus F}$, we have $\theta_{u_n}(f)=\gamma_n(f)$ (where $\gamma_n$ is the left inverse for $\alpha_n$ discussed at the start of the section), and hence \eqref{revcov} implies that $\mu_n^*\pi(f)\mu_n=\pi(\theta_{u_n}(f))$.

Now suppose that $F$ is a fixed finite set of primes containing the prime factors of $m$, and take $f\in C_F\cap C(\hatz)_\chi$. Then for each finite subset $E$ of $\primes\setminus F$,  \lemref{reconstructionlemma} implies that 
\begin{align}\label{calcuserecontrs}
\phi(\pi(f))&= \sum_{n\in \nx_\setb } \frac{n^{-\beta}}{\zeta_\setb (\beta)} \phi_{Q_\setb } (\mu^*_n\pi(f)\mu_n)\\ 
&=\sum_{n\in \nx_\setb } \frac{n^{-\beta}}{\zeta_\setb (\beta)} \phi_{Q_\setb } (\pi(\theta_{u_n}(f)))\notag\\
&= \phi_{Q_\setb } (\pi(f)) \sum_{n\in \nx_\setb } \frac{n^{-\beta}}{\zeta_\setb (\beta)} \chi(u_n).\notag
\end{align}
 Since $n \mapsto \chi(u_n)$ is a nontrivial Dirichlet character modulo $m$, we have 
\[
\sum_{n\in \nx_\setb } n^{-\beta} \chi(u_n) = \prod_{p\in \setb } (1 - p^{-\beta} \chi(u_p))\ \text{ for $\beta>0$;}
\]
as $\setb $ increases through a listing of $\primes \setminus F$, this product converges to 
 $\prod_{p\in \primes \setminus F} (1 - p^{-\beta} \chi(u_p))$, which is finite (by, for example, Theorem 5 on page 161 of \cite{lan}). On the other hand, since $\beta \leq 1$, we have
$\zeta_\setb (\beta) \to \zeta_{\primes \setminus F} (\beta) = \infty$ as $E$ increases. Thus \eqref{calcuserecontrs} implies that $\phi(\pi(f)) = 0$.
\end{proof}


\begin{thebibliography}{00}
\bibitem{bos-con} J.-B. Bost and A. Connes, \emph{Hecke algebras, type
III factors  and phase transitions with spontaneous symmetry breaking
in number theory}, Selecta Math. (New Series) {\bf 1} (1995), 411--457.

 \bibitem{bra-rob} O. Bratteli and D.W. Robinson, Operator Algebras and Quantum Statistical Mechanics 2, Second Edition, Springer-Verlag, Berlin, 1997.
 
 \bibitem{CM} A. Connes and M. Marcolli, \emph{Quantum statistical mechanics of $\Bbb{Q}$-lattices}, in Frontiers in Number Theory, Physics, and Geometry I,  Springer-Verlag, 2006, pages 269--349.

\bibitem{CM2} A. Connes and M. Marcolli, Noncommutative Geometry, Quantum Fields, and Motives, Colloquium Publications, Vol.55, American Mathematical Society, 2008.

\bibitem{CL1} 
{J. Crisp} and {M. Laca}, 
\emph{On the Toeplitz algebras of right-angled  and finite-type Artin groups},
{ J. Austral. Math. Soc.}  {\bf 72} (2002), 223--245. 

 \bibitem{CL2} J. Crisp and M. Laca, 
\emph{Boundary quotients and ideals of Toeplitz $C^*$-algebras of Artin groups},  J. Funct. Anal. {\bf 242} (2007), 127--156. 

\bibitem{cun2} 
{J. Cuntz}, \emph{$C^*$-algebras associated with the $ax+b$ semigroup over $\N$}, in $K$-Theory and Noncommutative Geometry (Valladolid, 2006), European Math. Soc, 2008, pages 201--215.

\bibitem{eva}
D.E. Evans, \emph{On ${\mathcal O}_n$}, Publ. Res. Inst. Math. Sci.
{\bf 16} (1980), 915--927. 

\bibitem{EL} R. Exel and M. Laca, \emph{Partial dynamical systems and the KMS condition}, Comm. Math. Phys. {\bf 232} (2003), 223--277.

\bibitem{topfree} 
{R. Exel}, {M. Laca} and {J. Quigg}, \emph{Partial dynamical systems and $C^*$-algebras generated by partial isometries}, J. Operator Theory {\bf 47} (2002), 169--186.

\bibitem{HL} D. Harari and E. Leichtnam, \emph{Extension du ph\'enom\`ene de brisure spontan\'ee de sym\'etrie de Bost-Connes au cas de corpes globals quelconques}, Selecta Math. (New Series) \textbf{3} (1997), 205--243.

\bibitem{diri}  M. Laca, \emph{Semigroups of *-endomorphisms, Dirichlet
series and  phase transitions}, J. Funct. Anal. {\bf 152} (1998),
330--378.

\bibitem{purelinf} M. Laca, \emph{Purely infinite simple Toeplitz algebras}, {J.
Operator Theory} {\bf 41} (1999), 421--435.

\bibitem{lacaN} M. Laca and S. Neshveyev, \emph{KMS states of quasi-free dynamics on Pimsner algebras}, J. Funct. Anal. {\bf 211} (2004), 457--482.

\bibitem{cmgl2} M. Laca, N.S. Larsen and S. Neshveyev, \emph{Phase transition in the Connes-Marcolli $\mathrm{GL}_2$-system}, J. Noncommut. Geom. {\bf 1} (2007), 397--430.

\bibitem{quasilat} 
{M. Laca} and {I. Raeburn}, \emph{Semigroup crossed products and the Toeplitz algebras of nonabelian groups}, J. Funct. Anal. {\bf 139} (1996), 415--440.
 
 \bibitem{bcalg}  M. Laca and I. Raeburn, \emph{A semigroup crossed
product arising in number theory},  J. London Math. Soc. {\bf 59} (1999), 330--344.

\bibitem{LvF} M. Laca and M. van Frankenhuijsen, \emph{Phase transitions on Hecke $C^*$-algebras and class field theory over $\Q$}, J. reine angew. Math. {\bf 595} (2006), 25--53.

\bibitem{lan} S. Lang, Algebraic Number Theory, Springer-Verlag, New York, 1986.
 
 \bibitem{nes} S. Neshveyev, \emph{Ergodicity of the action of the positive
rationals on the group of finite adeles and the Bost--Connes phase
transition theorem}, Proc. Amer. Math. Soc. {\bf 130} (2002),
2999--3003.

\bibitem{nica} 
{A. Nica}, \emph{$C^*$-algebras generated by isometries and Wiener-Hopf operators}, J. Operator Theory {\bf 27} (1992), 17--52.
 
 \bibitem{ped} G. K. Pedersen, $C^*$-Algebras and Their Automorphism Groups, 
 Academic Press, London, 1979.
 
\bibitem{rud} W. Rudin, Real and Complex Analysis, Third Edition, McGraw-Hill, New York, 1987.

\end{thebibliography}
\end{document}